\documentclass[reqno,11pt]{amsart}
\usepackage{amsmath,graphicx,amssymb}
\usepackage[all]{xy}
\usepackage{epic}
\usepackage{eepic}
\usepackage{hyperref}

\newtheorem{theorem}{Theorem}         % 1st argument is your name for it
\newtheorem{proposition}{Proposition}%[section] % 2nd argument is what is printed
\newtheorem{cor}[proposition]{Corollary}
\newtheorem{lemma}[proposition]{Lemma}
\newtheorem{remark}[proposition]{Remark}

\numberwithin{equation}{section}

\renewcommand{\AA}{\mathcal A}
\newcommand{\Z}{\mathbb Z}
\newcommand{\N}{\mathbb N}
\newcommand{\R}{\mathbb R}

\newcommand{\Q}{\mathbb Q}
\newcommand{\RRR}{\mathfrak R}
\newcommand{\SSS}{\mathfrak S}

\newcommand{\EE}{\mathcal E}

\renewcommand{\H}{\mathbb H}

\newcommand{\CC}{\mathcal C}

\newcommand{\RR}{\mathcal R}

\newcommand{\MM}{\mathcal M}
\newcommand{\wMM}{\widetilde{\RRR}}
\newcommand{\NN}{\mathcal N}
\newcommand{\wNN}{\widetilde{\NN}}
\newcommand{\TT}{\mathcal T}
\newcommand{\FF}{\mathcal F}

\renewcommand{\Re}{\operatorname{Re}}
\newcommand{\veps}{\varepsilon}
\newcommand{\tr}{\operatorname{Tr}}
\newcommand{\comment}[1]{}

\title[Angles between reciprocal geodesics]{Pair correlation of
angles between reciprocal \\  geodesics on the modular surface}
\author{Florin P. Boca, Vicen\c tiu Pa\c sol, Alexandru A. Popa, Alexandru Zaharescu}

\address{FPB and AZ: Department of Mathematics, University of Illinois, 1409 W. Green Street, Urbana, IL 61801, USA}

\address{E-mail: fboca@illinois.edu}

\address{E-mail: zaharesc@illinois.edu}

\address{FPB, VP, AAP, AZ: Institute of Mathematics ``Simion Stoilow" of the Romanian Academy,
P.O. Box 1-764, RO-014700 Bucharest, Romania}

\address{E-mail: vicentiu.pasol@imar.ro}

\address{E-mail: alexandru.popa@imar.ro}


\date{October 21, 2013}

\begin{document}

%\maketitle

\begin{abstract}
The existence of the limiting pair correlation for angles between reciprocal geodesics on the
modular surface is established. An explicit formula is provided, which
captures geometric information about the length of reciprocal geodesics, as well as
arithmetic information about the associated reciprocal classes of binary quadratic forms.
One striking feature is the absence of a gap beyond zero in
the limiting distribution, contrasting with the analog Euclidean situation.
\end{abstract}

\maketitle

\section{Introduction}
Let $\H$ denote the upper half-plane and $\Gamma=\mathrm{PSL}_2(\Z)$ the modular group.
Consider the modular surface $X=\Gamma\backslash \mathbb{H}$, and let $\Pi:\H \rightarrow
X$ be the natural projection. The angles on the upper half plane $\H$ considered in this
paper are the same as the angles on $X$ between the closed geodesics passing through
$\Pi(i)$, and the image of the imaginary axis. These geodesics were first introduced in
connection with the associated ``self-inverse classes'' of binary quadratic forms in the
classical work of Fricke and Klein \cite[p.164]{FK}, and the primitive geodesics among
them were studied recently and called reciprocal geodesics by Sarnak \cite{Sar}. The aim
of this paper is to establish the
existence of the pair correlation measure of their angles and to explicitly express it.

For $g\in \Gamma$, denote by $\theta_g \in [-\pi,\pi]$ the angle between the
vertical geodesic $[i,0]$ and the geodesic ray $[i,gi]$. For $z_1,z_2\in\H$, let $d(z_1,z_2)$ denote the
hyperbolic distance, and set $$\| g\|^2=2\cosh d(i,gi)=a^2+b^2+c^2+d^2, \quad \text{for } g=\left(
\begin{smallmatrix} a & b \\
c & d \end{smallmatrix}\right) \in \mathrm{SL}_2(\R).$$

It was proved by Nicholls \cite{Ni1} (see also
\cite[Theorem 10.7.6]{Ni2}) that for any  discrete subgroup $\Gamma$ of finite covolume in $\mathrm{PSL}_2(\R)$,
the angles
$\theta_\gamma$ are uniformly distributed, in the sense that for any fixed interval
$I\subseteq [-\pi,\pi]$,
\begin{equation*}
\lim_{R\rightarrow \infty} \frac{ \# \big\{ \gamma\in\Gamma:
\theta_\gamma \in I, d(i,\gamma i) \leqslant R\big\}}{
\# \big\{ \gamma\in \Gamma: d(i,\gamma i)\leqslant R\big\}} = \frac{\vert I\vert}{2\pi} .
\end{equation*}
Effective estimates for the rate of convergence that allow one to take $\vert I\vert \asymp e^{-cR}$
as $R\rightarrow \infty$ for some constant $c=c_\Gamma >0$ were proved for
$\Gamma=\Gamma (N)$
by one of us \cite{Bo}, and in general situations by Risager and Truelsen \cite{RT} and by
Gorodnik and Nevo \cite{GN}.
Other related results concerning the uniform distribution of real parts of orbits in
hyperbolic spaces were
proved by Good \cite{Goo}, and more recently by Risager and Rudnick \cite{RR}.

The statistics of spacings, such as the pair correlation or the nearest neighbor
distribution
(also known as the gap distribution) measure the fine structure of sequences of real
numbers in a more
subtle way than the classical Weyl uniform distribution.
Very little is known about the spacing statistics of closed geodesics.
In fact, the only result that we are aware of, due to Pollicott and Sharp \cite{PS},
concerns the correlation of differences of lengths of pairs of
closed geodesics on a compact surface of negative curvature,
ordered with respect to the word length on the fundamental group.

This paper investigates the pair correlation
of angles $\theta_\gamma$ with $d(i,\gamma i)\leqslant R$, or equivalently
with $\| \gamma\|^2 \leqslant Q^2 =e^R \sim 2\cosh R$ as $Q\rightarrow\infty$.
As explained in Section \ref{Section2}, these are exactly the angles between
reciprocal geodesics on the modular surface.

The Euclidean analog of this problem considers the angles between the line segments connecting
the origin $(0,0)$ with all integer points $(m,n)$ with $m^2+n^2 \leqslant Q^2$ as $Q\rightarrow \infty$.
When only primitive lattice points are being considered (rays are counted with multiplicity one),
the problem reduces to the study of the pair correlation of the sequence of Farey
fractions with the $L^2$ norm $\| m/n\|_2^2=m^2+n^2$.
Its pair correlation function is plotted on the left of Figure \ref{Figure1}.
When Farey fractions are ordered by their denominator, the pair correlation is shown to
exist and it is explicitly computed in \cite{BZ}.
A common important feature is the existence of a gap
beyond zero for the pair correlation function. This is
an ultimate reflection of the fact that the area of a nondegenerate triangle with
integer vertices is at least $\frac{1}{2}$, which corresponds to the familiar inequality
$| \frac{b}{d}-\frac{a}{c} | \geqslant \frac{1}{cd}$ satisfied by two lattice points
$P=(a,b)$ and $Q=(c,d)$ with $\operatorname{Area} (\triangle OPQ)>0$.

For the hyperbolic lattice centered at $i$, it is convenient to start with the
(non-uniformly distributed) numbers
$\tan( \frac{\theta_\gamma}{2})$ with multiplicities, rather than
the angles $\theta_\gamma$ themselves. Employing obvious symmetries explained in Section
\ref{Section3}, it is further convenient to restrict to a set of representatives
$\Gamma_{\bf I}$ consisting of matrices $\gamma$ with nonnegative entries such that the
point $\gamma i$ is in
the first quadrant in Figure \ref{Figure2}.
The pair correlation measures of the finite set ${\mathfrak A}_Q$ of elements
$\theta_\gamma$ with $\gamma \in \Gamma_{\bf I}$ and $\| \gamma \| \leqslant Q$ (counted
with multiplicities) is defined as
\begin{equation*}
%\begin{split}
R_Q^{\mathfrak A} (\xi)  =\frac{1}{B_Q} \# \bigg\{ (\gamma,\gamma^\prime)\in\Gamma_{\bf I}^2 :
\| \gamma\|,\| \gamma^\prime \| \leqslant Q,\ \gamma^\prime \neq \gamma,\ 0\leqslant
\frac{2}{\pi} \big( \theta_{\gamma^\prime}- \theta_\gamma \big) \leqslant \frac{\xi}{B_Q}\bigg\},
%\end{split}
\end{equation*}
where $B_Q \sim \frac{3}{8}Q^2$ denotes the number of elements $\gamma \in\Gamma_{\bf I}$
with $\| \gamma\|\leqslant Q$. As it will be used in the proof, we similarly define the
pair correlation measure $R_Q^{\mathfrak T} (\xi)$ of the set ${\mathfrak T}_Q$ of
elements $\tan (\frac{\theta_\gamma}{2})$ with $\gamma \in \Gamma_{\bf I}$ and $\| \gamma
\| \leqslant Q$.

One striking feature, illustrated by the numerical calculations in Figure \ref{Figure1},
points to the absence of a gap beyond zero in the limiting distribution, in
contrast with the analog Euclidean situation.

The main result of this paper is the proof of existence and explicit computation of the
pair correlation measure $R_2^{\mathfrak A}$ given by
\begin{equation}\label{1.1}
R_2^{\mathfrak A} (\xi) = R_2^{\mathfrak A} \big( (0,\xi]\big): =\lim\limits_{Q\rightarrow
\infty} R_Q^{\mathfrak A} (\xi),
\end{equation}
and similarly for $R_2^{\mathfrak T}$, thus answering a question raised in \cite{Bo}.

To give a precise statement consider ${\mathfrak S}$, the free
semigroup on two generators $L=\left( \begin{smallmatrix} 1 & 0 \\ 1 & 1 \end{smallmatrix}\right)$ and
$R=\left( \begin{smallmatrix} 1 & 1 \\ 0 & 1 \end{smallmatrix}\right)$.
Repeated application of the Euclidean algorithm shows that
${\mathfrak S}\cup \{ I\}$ coincides with the set of matrices in $\mathrm{SL}_2(\Z)$ with nonnegative entries.
The explicit formula for  $R_2^{\mathfrak T} (\xi)$ is given as a series of volumes summed over $\SSS$, plus a
finite sum of volumes, and it is stated in Theorem \ref{T2} of Section 7. The formula for $R_2^{\mathfrak T} (\xi)$ leads to an
explicit formula for $R_2^{\mathfrak A} (\xi)$, which we state here, partly because the pair correlation
function for the angles $\theta_\gamma$ is more interesting, being equidistributed, and partly because the
formula we obtain is simpler.

\begin{theorem}\label{T1}
The pair correlation measure $R_2^{\mathfrak A}$ on $[0,\infty)$ exists and is given by the $C^1$
function
\begin{equation}\label{1.2}
R_2^{\mathfrak A} \bigg( \frac{3}{4\pi}\xi\bigg) = \frac{8}{3\zeta (2)} \Bigg(
\sum\limits_{M\in {\mathfrak S}} B_M(\xi) +\sum_{\ell\in [0,\xi/2)}
\sum_{K\in [1,\xi/2)} A_{K,\ell}(\xi) \Bigg).
\end{equation}
For $M\in\SSS$, letting $U_M= \| M\|^2/\sqrt{\| M\|^4-4}$, $\theta_M$ as above, and $f_+=\max(f,0)$, we
have
\[
B_M(\xi)=\frac{\pi}{4} \int_0^{\pi /2}
\frac{\left(1/\sqrt{\| M\|^4-4}-\sin(2\theta-\theta_M)/\xi\right)_+}{U_M+\cos(2\theta-\theta_M)}\, d\theta.
\]
For integers $\ell\in [0,\frac{\xi}{2})$, $K\in [1,\frac{\xi}{2})$, we have
\[ A_{K,\ell}(\xi)= \int_0^{\pi/4} A_{K,\ell}\bigg(\frac{\xi}{2\cos^2 t},t\bigg)\, \frac{dt}{\cos^2 t} ,
\]
where $A_{K,\ell} (\xi,t)$ is the area of the region defined by
\begin{equation}\label{1.3}
\left\{ re^{i\theta}\in[0,1]^2:L_{\ell+1}(e^{i\theta})>0, \
\frac{F_{K,\ell}(\theta)}{\xi}  \leqslant  r^2 \leqslant
\frac{\cos^2 t}{ \max\big\{ 1,L_\ell^2 (e^{i\theta}) +L_{\ell+1}^2 (e^{i\theta}) \big\}}
\right\},
\end{equation}
with $e^{i\theta}=(\cos\theta,\sin\theta)$, the piecewise linear functions $L_i$ as defined in \eqref{5.5}, and with
\begin{equation*}
F_{K,\ell} (\theta) := \cot\theta +\sum_{i=1}^\ell
\frac{1}{L_{i-1}(e^{i\theta})L_{i}(e^{i\theta})}
+\frac{L_{\ell+1}(e^{i\theta})}{L_\ell (e^{i\theta}) \big( L_\ell^2
(e^{i\theta})+L_{\ell+1}^2 (e^{i\theta})\big)} .
\end{equation*}
\end{theorem}

Rates of convergence in \eqref{1.1} are effectively described in the proof of Theorem \ref{T2} and in Proposition
\ref{P15}.

When $\xi\leqslant 2$, the second sum in \eqref{1.2} disappears and the derivative $B_M'(\xi)$ is explicitly
computed in Lemma \ref{L17}, yielding an explicit formula for the pair correlation density function $g_2^{\mathfrak
A}(\xi)=\frac{ d R_2^{\mathfrak A} (\xi)}{d\xi} $ which matches the graph in Figure \ref{Figure1}.

\begin{cor}\label{C1}
For $0<\xi\leqslant 2$ we have
\[
g_2^{\mathfrak A}\left( \frac{3}{4\pi}\xi
\right)=\frac{16}{3\xi^2}\sum_{M\in
\mathfrak{S}} \ln\left(
\frac{\|M\|^2+\sqrt{\|M\|^4-4}}{\|M\|^2+\sqrt{\|M\|^4-4-\xi^2}}\right) .
\]
A formula valid for $0<\xi\leqslant 4$ is given in \eqref{8.13} after computing $A_{0,K}^{\prime}(\xi)$.
\end{cor}
The computation is performed in \S\ref{sec8.1}, and it identifies the first spike in the graph of $g_2^{\mathfrak
A}(x)$ at $x=\frac{3}{4\pi} \sqrt{5}$. A proof of an explicit formula for the pair correlation density
$g_2^{\mathfrak A}(x)$ valid for all $x$, and working also when the point $i$ is replaced by the other
elliptic point $\rho=e^{\pi i/3}$, will be given in \cite{BPZ}.

\begin{figure}[ht]
\begin{center}
\includegraphics*[scale=0.42, bb=8 7 557 414]{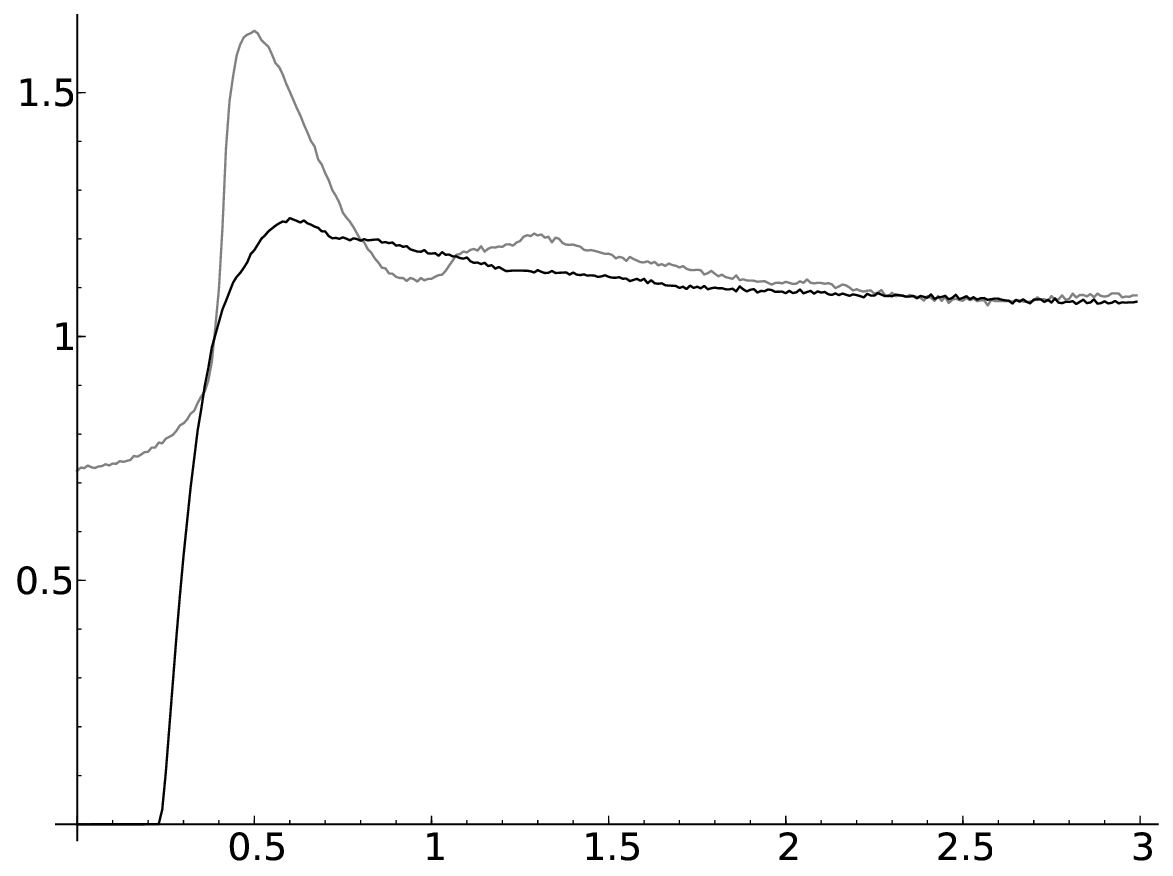}%bb=0 0 500 320
\includegraphics*[scale=0.42, bb=8 7 557 414]{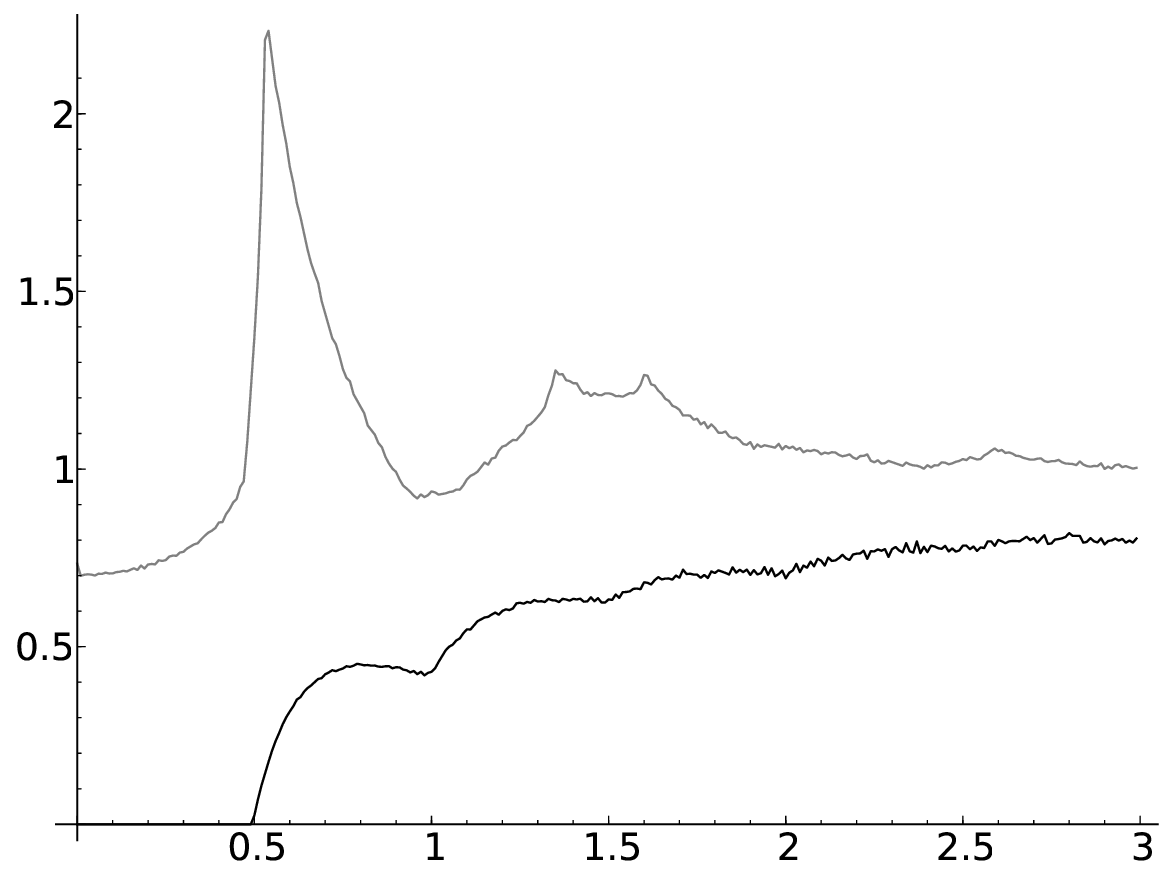}
\caption{The pair correlation functions $g_2^{\mathfrak T}$ (left) and
$g_2^{\mathfrak A}$ (right), plotted in grey, compared with the pair
correlation function of Farey fractions with $L^2$ norm (left), and
of the angles (with multiplicities) of lattice points in Euclidean balls (right). The
graphs are obtained by counting the pairs in their definition, using $Q=4000$ for which
$B_Q=6000203$. We used Magma \cite{BCP} for the numerical computations, and SAGE \cite{St}
for plotting the graphs.  }
\label{Figure1}
\end{center}
\end{figure}

Since the series in Corollary \ref{C1} is dominated by the absolutely convergent $\sum_{M}\xi^2 \|M\|^{-4}$,
we can take the limit as $\xi\rightarrow 0$:
\[
g_2^{\mathfrak A}(0) = \frac{2}{3} \sum\limits_{M\in
{\mathfrak S}}  \left(\frac{\|M\|^2}{\sqrt{\|M\|^4-4}}-1\right)=0.7015...
\]
Remarkably the previous two formulas, as well as \eqref{1.2} for $\xi \leqslant 2$, can be written geometrically as
a sum over the primitive closed geodesics
$\mathcal{C}$ on $X$ which pass through the point $\Pi(i)$, where the summand depends only
on the length $\ell(\mathcal{C})$:
\[
g_2^{\mathfrak A}(0)=\frac{8}{3} \sum_{\mathcal{C}}
\sum_{n\geqslant 1} \frac{1}{e^{n\ell(\mathcal{C})}-1}.
\]
This is proved in Section \ref{Section2}, where we also give an arithmetic version
based on an explicit description of the reciprocal geodesics $\mathcal{C}$ due to Sarnak
\cite{Sar}.

For the rest of the introduction we sketch the main ideas behind the proof, describing also the organization of
the article. After reducing to angles in the first quadrant in Section \ref{Section3}, we show that the pair
correlation of the quantities $\Psi(\gamma)=\tan(\frac{\theta_\gamma}{2})$ is identical to that of $\Phi(\gamma)=\Re(
\gamma i)$. We are led to estimating the cardinality of the set
\[
\big\{ (\gamma,\gamma^\prime)\in\Gamma_{\bf I}^2 :
\| \gamma\|,\| \gamma^\prime\| \leqslant Q,\  \gamma^\prime \neq \gamma,\ 0\leqslant
Q^2 (\Phi(\gamma^\prime)-\Phi(\gamma)) \leqslant \xi\big\}.
\]
For $\gamma=\left( \begin{smallmatrix} p' & p \\ q' & q \end{smallmatrix}\right)$ with nonnegative
entries, $\|\gamma\|\leqslant Q$,
and $q,q'>0$, consider the associated Farey interval $[\frac{p}{q},\frac{p'}{q'}]$, which contains
$\Phi(\gamma)$. In Section \ref{Section4}, we break the set of pairs $(\gamma, \gamma')$ above in two parts,
depending on whether one of the associated Farey intervals contains the other, or the two intervals intersect at most
at one endpoint. In the first case we have $\gamma=\gamma' M$ or $\gamma'=\gamma M$ with $M\in \SSS$, while in the
second we have a similar relation depending on the number $\ell$ of consecutive Farey fractions there are
between the two intervals. The first case contributes to the series over $\SSS$ in \eqref{1.2}, while the second
case contributes to the sum over $K,\ell$. The triangle map $T$ whose iterates define the piecewise linear
functions $L_i(x,y)$, first introduced in \cite{BCZ}, makes its
appearance in the second case, being related to the denominator of the successor function for Farey fractions.

To estimate the number of pairs $(\gamma, \gamma M)$ in the first case, a key observation is that for each
$M\in\Gamma$ there exists an explicit elementary function $\Xi_M(x,y)$, given by \eqref{5.1}, such that
\[
\Phi(\gamma)-\Phi(\gamma M)=\Xi_M(q',q)
\]
for $\gamma$ as above. Together with estimates for the number of points in two dimensional regions based on bounds
on Kloosterman sums (Lemma \ref{L7}), this allows us to estimate the number of pairs $(\gamma, \gamma M)$
with fixed $M\in \SSS$, in terms of the volume of a three dimensional body $S_{M,\xi}$ given in
\eqref{7.15}. The absence of a gap beyond zero in the pair correlation measure arises as a result of this estimate.
The details of the calculation are given in Section \ref{Section7}, leading to an explicit formula for $R_2^{\mathfrak
T}$ (Theorem \ref{T2}).

Finally in Section \ref{Section8} we pass to the pair correlation of the angles $\theta_\gamma$, obtaining the
formulas of Theorem \ref{T1} and Corollary \ref{C1}.

In this paper we focus on the full modular lattice centered at $i$, both because of the
arithmetic connection with reciprocal geodesics, and because in this case the connection
between unimodular matrices and Farey intervals is most transparent. It is this connection
and the intuition provided by the repulsion of Farey fractions that guides our argument,
and leads to the explicit formula for the pair correlation function, which is the first of
this kind for hyperbolic lattices.

In a subsequent paper \cite{BPZ}, we abstract some of this intuition and propose a
\emph{different} conjectural formula for the pair correlation function of an arbitrary
lattice in $\mathrm{PSL}_2(\R)$, centered at a point on the upper half plane, which we
prove for the full level lattice centered at elliptic points. While the formula in
\cite{BPZ} is more general, the method of proof, and the combinatorial-geometric intuition
behind it, is reflected more accurately in the formula of Theorem \ref{T1}: the infinite
sum in the formula corresponds to pairs of matrices where there is no repulsion between
their Farey intervals, while the finite sum corresponds to pairs of matrices where there
is repulsion. The approach used in \cite{BPZ} builds on the estimates and method of the
present paper.

A proof of the conjecture in \cite{BPZ} by spectral methods has been proposed very
recently by Kelmer and Kontorovich in the preprint \cite{KK}. By comparison, our approach
is entirely elementary (using only standard bounds on Kloosterman sums), and via the
repulsion argument it provides a natural way of approximating the pair correlation
function. A key insight in the present paper, which is also
the starting point of \cite{BPZ} and \cite{KK}, is that instead of counting pairs
$(\gamma,\gamma')\in \Gamma \times \Gamma$ in the definition of the pair correlation
measure, we fix a matrix $M$, count pairs $(\gamma,\gamma M)$, and sum over $M$. The same
approach may prove useful for the pair correlation problem for lattices in other groups as
well.

\section{Reciprocal geodesics on the modular surface}\label{Section2}

In this section we recall the definition of reciprocal geodesics and explain how the pair
correlation of the angles they make with the imaginary axis is related to the pair
correlation considered in the introduction. We also show that the sums over the
semigroup $\mathfrak{S}$ appearing in the introduction can be expressed geometrically in
terms of sums over primitive reciprocal geodesics. A
description of the trajectory of reciprocal geodesics  on the fundamental domain
seems to have first appeared in the classical work of Fricke and Klein \cite[p.164]{FK},
where it is shown that they consist of two closed loops, one the reverse of the other.
There the terminology ``sich selbst inverse Classe'' is used for the equivalence classes
of quadratic forms corresponding to reciprocal conjugacy classes of hyperbolic matrices.

Oriented closed geodesics on $X$ are in one-to-one correspondence with conjugacy classes
$\{\gamma\}$ of hyperbolic elements $\gamma\in \Gamma$.
To a hyperbolic element $\gamma\in \Gamma$ one attaches its axis $a_\gamma$ on $\H$,
namely the semicircle whose endpoints are the fixed points of $\gamma$ on the real
axis. The part of the semicircle between $z_0$ and $\gamma z_0$, for any $z_0\in
a_\gamma$, projects to a closed geodesic on $X$, with multiplicity one if only if $\gamma$
is a primitive matrix (not a power of another hyperbolic element of $\Gamma$). The group
that fixes the semicircle $a_\gamma$ (or equivalently its endpoints on the real axis)
is generated by one primitive element $\gamma_0$.

We are concerned with (oriented) closed geodesics passing through $\Pi (i)$ on $X$. Since
the axis of a hyperbolic element $A$ passes through $i$ if and only if $A$ is symmetric,
the closed geodesics passing through $\Pi (i)$ correspond to the set $\RR$ of
hyperbolic conjugacy classes $\{\gamma\}$ which contain a symmetric matrix. The latter are
exactly the reciprocal geodesics considered by Sarnak in \cite{Sar}, where only
primitive geodesics are considered.

The reciprocal geodesics can be parameterized in a two-to-one manner by the set  ${\mathfrak
S}\subset\Gamma$, defined in the introduction, which consists of matrices distinct
from the identity with nonnegative entries. To describe this correspondence, let
$\AA\subset\Gamma$ be the set of symmetric hyperbolic matrices with positive entries.
Then we have maps
\begin{equation}\label{2.1}
{\mathfrak S}\rightarrow \AA \rightarrow \RR
\end{equation}
where the first map takes $\gamma\in {\mathfrak S}$ to $A=\gamma \gamma^t$, and the second
takes the hyperbolic symmetric $A$  to its conjugacy class $\{A\}$. The first map is
bijective, while the second is two-to-one and onto, as it follows from \cite{Sar}. More
precisely, if $A=\gamma\gamma^t\in \AA$ is a primitive matrix, then $B=\gamma^t \gamma\ne
A$ is the only other matrix in $\AA$ conjugate with $A$, and $\{A^n\}=\{B^n\}$ for all
$n\geqslant 0$.

Note also that $\| \gamma \|^2=\tr(\gamma\gamma^t)$, and if $A$ is hyperbolic
with $\tr(A)=T$, then the length of the geodesic associated to $\{A\}$ is $2\ln N(A)$
with $N(A)=\frac{1}{2} (T+\sqrt{T^2-4})$.

We need the following:
\begin{lemma}\label{L2}
Let $A\in \Gamma$ be a hyperbolic symmetric matrix and let $\gamma \in \Gamma$ such that
$A=\gamma \gamma^t$. Then the point $\gamma i$ is halfway
{\em (}in hyperbolic distance{\em )} between $i$ and $Ai$ on the axis of $A$.
\end{lemma}

\begin{proof}
We have $d(i,\gamma i)= d(i, \gamma^t i)=d(\gamma i, A i)$ where the first equality
follows from the hyperbolic distance formula and the second since $\Gamma$ acts by
isometries on $\H$.
Using formula \eqref{3.3}, one checks that the angles of $i, \gamma i$ and $i, Ai$  are
equal, hence $\gamma i$ is indeed on the axis of $A$.
\end{proof}

We can now explain the connection between the angles $\theta_\gamma$ in the first and
second quadrant in Figure \ref{Figure2}, and the angles made by the reciprocal geodesics
with the image $\Pi(i\rightarrow i\infty)=\Pi(i\rightarrow 0)$. Namely, points in the
first and second quadrant are parameterized by $\gamma i$ with $\gamma \in {\mathfrak S}$,
and by the lemma the reciprocal geodesic corresponding to $A=\gamma \gamma^t\in \AA$
consists of the loop $\Pi(i\rightarrow \gamma i)$, followed by $\Pi(i\rightarrow \gamma^t
i)$ (which is the same as the reverse of the first loop).
Therefore to each reciprocal geodesic corresponding to $A=\gamma \gamma^t\in\AA$
correspond two angles, those attached to $\gamma i$ and $\gamma^t i$ in Figure
\ref{Figure2}, measured in the first or second
quadrant so that all angles are between 0 and $\frac{\pi}{2}$.

In conclusion the angles made by the reciprocal geodesics on $X$ with the
fixed direction $\Pi(i\rightarrow i\infty)$ consist of the angles
in the first quadrant considered before, each appearing twice. Ordering the points $\gamma
i$ in the first quadrant by $\|\gamma\|$ corresponds to ordering the geodesics by
their length. Therefore the pair correlation measure of the angles of reciprocal
geodesics is $2 R_2^{\mathfrak A}(\frac{\xi}{2})$, where $R_2^{\mathfrak A}$ was defined in the
introduction.

The parametrization \eqref{2.1} of reciprocal geodesics allows one to rewrite the
series appearing in the formula for $g_2^{\mathfrak A} (0)$ in the introduction, as a series over the
primitive reciprocal classes $\RR^{\operatorname{prim}}$:
\[
\sum\limits_{M\in
{\mathfrak S}} \left(\frac{\|M\|^2}{\sqrt{\|M\|^4-4}}-1\right)= \sum\limits_{A\in\AA}
\frac{2}{N(A)^2-1}=4\sum\limits_{\{\gamma\}\in \RR^{\operatorname{prim}}}
\sum_{n\geqslant 1}\frac{1}{N(\gamma)^{2n}-1},
\]
where we have used the fact that for a hyperbolic matrix $A$ of trace $T$ we have
\[ \sqrt{T^2-4}=N(A)-N(A)^{-1}  \text{ and  }  N(A^n)=N(A)^n.\]

One can rewrite the sum further using the arithmetic description of primitive reciprocal
geodesics given in \cite{Sar}. Namely, let $\mathcal{D}_\RR$ be the set of nonsquare
positive discriminants $2^\alpha D'$ with $\alpha\in\{0,2,3\}$ and $D'$ odd divisible only
by primes $p\equiv 1 \pmod{4}$. Then the set of primitive reciprocal classes
$\RR^{\operatorname{prim}}$ decomposes as a disjoint union of finite sets:
\[
\RR^{\operatorname{prim}}=\bigcup_{d\in \mathcal{D}_\RR} \RR^{\operatorname{prim}}_d
\]
with $|\RR^{\operatorname{prim}}_d|= \nu(d)$, the number of genera of binary
quadratic forms of discriminant $d$. For $d\in \mathcal{D}_\RR$, $\nu(d)$ equals
$2^{\lambda-1}$, or respectively $2^{\lambda}$ depending on whether $8\nmid d$, or
respectively $8\mid d$, and $\lambda$ is the number of distinct odd prime factors of $d$.
Each class $\{\gamma\} \in \RR^{\operatorname{prim}}_d$ has
\[
N(\gamma)=\alpha_d=\textstyle\frac{1}{2} (u_0+v_0\sqrt{d})
\]
with $(u_0,v_0)$ the minimal positive solution
to Pell's equation $u^2-d v^2=4$. We then have
\[
\sum\limits_{\{\gamma\}\in \RR^{\operatorname{prim}}}
\sum_{n\geqslant 1}\frac{1}{N(\gamma)^{2n}-1}=\sum\limits_{d\in \mathcal{D}_\RR }
\sum_{n\geqslant 1}\frac{\nu(d)}{\alpha_d^{2n}-1}.
\]

In the same way, by Lemma \ref{L13} the pair correlation measure $R_2^{\mathfrak T} (\xi)$ in Theorem 1
can be written for
$\xi\leqslant 1$ as a sum over classes $\{\gamma\}\in
\RR^{\operatorname{prim}}$, where each summand depends only on $\xi$ and $N(\gamma)$.

\section{Reduction to the first quadrant}\label{Section3}

In this section we establish notation in use throughout the paper, and we reduce the pair correlation problem to
angles in the first quadrant. A similar reduction can be found in \cite{Ch}, in the context of visibility problems
for the hyperbolic lattice centered at $i$.

For each $g=\left( \begin{smallmatrix} a & b \\ c & d \end{smallmatrix}\right) \in SL_2 (\R)$ define
the quantities
\begin{equation}\label{3.1}
\begin{split}
& X_g=a^2+b^2 ,\qquad Y_g=c^2+d^2,\qquad Z_g =ac+bd, \qquad
T_g =X_g+Y_g=\| g\|^2 , \\
& \Phi (g)=\Re (gi)=\frac{Z_g}{Y_g} ,\qquad
\epsilon_g=\epsilon_{T_g} =\textstyle\frac{1}{2} \Big( T_g-\sqrt{T_g^2-4}\, \Big) .
\end{split}
\end{equation}

%The identities
%\begin{equation}\label{3.2}
%X_gY_g-Z_g^2=1
%\end{equation}
%and $\epsilon_g^2 -T_g \epsilon_g +1=0$ and formula \eqref{1.1} provide
%\begin{equation}\label{3.3}
%\Psi (g):= \tan \bigg( \frac{\theta_g}{2}\bigg) =\frac{X_g -\epsilon_g}{Z_g} =\frac{Z_g}{Y_g-\epsilon_g} .
%\end{equation}

The upper half-plane $\H$ is partitioned into the following four quadrants:
\begin{equation*}
\begin{split}
{\bf I} & = \{ z\in \H: \Re z >0, \vert z\vert < 1\}, \quad
{\bf II} =\{ z\in \H: \Re z >0, \vert z\vert >1\},\\
{\bf III} & =\{ z\in\H: \Re z<0, \vert z\vert >1\} ,\quad
{\bf IV}= \{ z\in\H: \Re z <0, \vert z\vert <1\} .
\end{split}
\end{equation*}
Note that all the points $g i$ for $g\in \Gamma$ lie in one of the four \emph{open}
quadrants, with the exception of $i$ itself. This follows from the relation
\begin{equation}\label{3.2}
X_gY_g-Z_g^2=1,
\end{equation}
which will be often used.

In this paragraph simply take $X=X_g$, $Y=Y_g$, $Z=Z_g$, $\theta=\theta_g$.
A direct calculation shows that the center of the circle through $i$ and $gi$ is
$\alpha=\frac{X-Y}{2Z}$, leading to
\begin{equation*}
\tan \theta_g =-\frac{1}{\alpha}=\frac{2Z_g}{Y_g-X_g} ,\quad \forall \theta_g \in [-\pi,\pi].
\end{equation*}
Plugging this into $\tan (\frac{\theta}{2})=\frac{\tan\theta}{1+\sqrt{1+\tan^2\theta}}$ if
$|\theta | <\frac{\pi}{2}$ and respectively $\tan( \frac{\theta}{2})=-\frac{1+\sqrt{1+\tan^2 \theta}}{\tan\theta}$
if $\frac{\pi}{2} < |\theta| <\pi$, and employing \eqref{3.2},
$| gi| <1 \Longleftrightarrow X<Y$, and $\Re (\gamma i) > 0 \Longleftrightarrow Z>0$,
we find the useful formulas
\begin{equation}\label{3.3}
\Psi (g):=\tan \bigg(
\frac{\theta_g}{2}\bigg) =\frac{\sqrt{T_g^2-4}+X_g-Y_g}{2Z_g} =\frac{X_g -\epsilon_g}{Z_g} =\frac{Z_g}{Y_g-\epsilon_g},
\quad \forall \theta_g \in [-\pi,\pi].
\end{equation}

We denote
$\gamma =\left(\begin{smallmatrix} a & b \\ c & d \end{smallmatrix}\right)$,
$\widetilde{\gamma} =\left(\begin{smallmatrix} d & c \\ b & a \end{smallmatrix}\right)$,
$s=\left(\begin{smallmatrix} 0 & -1 \\ 1 & 0 \end{smallmatrix}\right)$.
Let $\gamma\in \Gamma$, $\gamma \neq I,s$. For $\gamma i$ to be in the right half-plane we
need $\Re (\gamma i) >0$.
This is equivalent with $ac+bd >0$ and implies $ac\geqslant 0$, $bd \geqslant 0$ because
$abcd=bc+(bc)^2 \geqslant 0$.
Since $ac \geqslant 0$ without loss of generality we will assume $a\geqslant 0$ and
$c\geqslant 0$
(otherwise consider $-\gamma$ instead).
Without loss of generality assume $b\geqslant 0$, $d\geqslant 0$ as well
(otherwise can consider
$-\gamma s=\left( \begin{smallmatrix} -b & a \\ -d & c \end{smallmatrix}\right)$
instead since $\gamma i=\gamma si$), so we can assume $\gamma$ has only nonnegative
entries.

\begin{figure}
\begin{center}
\includegraphics{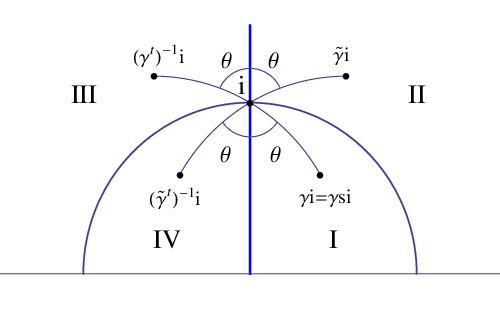}
\caption{Two symmetric geodesics through $i$}
\label{Figure2}
\end{center}
\end{figure}

If $a,b,c,d \geqslant 0$ and $ad-bc=1$, then $\frac{c}{a}$ and $\frac{d}{b}$ are both
$\leqslant 1$ or both $\geqslant 1$ (since open intervals between consecutive Farey fractions are either
nonintersecting or one contains the other). Since $\gamma i \in {\bf I} \Longleftrightarrow
a^2 +b^2 <c^2 +d^2$,
it follows that both $\frac{a}{c}$ and $\frac{b}{d}$ are $\leqslant 1$ for
$\gamma i\in {\bf I}$.
We conclude that among the eight matrices $\pm \gamma$, $\pm \gamma s$, $\pm
\widetilde{\gamma}$,
$\pm \widetilde{\gamma} s$, which have symmetric angles (see Figure \ref{Figure2}), the
one for which $\gamma i$ is in quadrant ${\bf I}$ can be chosen such that
\begin{equation*}
a,b,c,d\geqslant 0\quad \mbox{\rm and} \quad
0\leqslant \frac{b}{d} <\frac{a}{c} \leqslant 1
\end{equation*}
The set of such matrices $\gamma$ is denoted $\Gamma_{\bf I}$.

Consider the subset $\RRR_Q$ of $\Gamma_{\bf I}$ consisting of matrices with entries at most $Q$, that is
\begin{equation*}
\RRR_Q:=\left\{ \left( \begin{matrix} p^\prime & p \\ q^\prime & q
\end{matrix}\right) \in \Gamma :
0\leqslant p,p^\prime,q,q^\prime \leqslant Q,\ \frac{p}{q} <\frac{p^\prime}{q^\prime}
\leqslant 1\right\} ,
\end{equation*}
and its subset $\wMM_Q$ consisting of those $\gamma$ with
$\|\gamma\|\leqslant Q$. The cardinality $B_Q$ of $\wMM_Q$ is estimated in Corollary \ref{C8} as
$B_Q \sim \frac{3Q^2}{8}$, in agreement with formula (58) in \cite{Sar} for the number of reciprocal geodesics of
length at most $x= Q^2$.

Let $\FF_Q$ be the set of Farey fractions $\frac{p}{q}$ with $0\leqslant p \leqslant
q\leqslant  Q$ and $(p,q)=1$. The \emph{Farey tessellation} consists of semicircles on the
upper half plane connecting Farey fractions $0\leqslant
\frac{p}{q} < \frac{p'}{q'} \leqslant 1$ with $p'q-pq'=1$. We associate to matrices
$\gamma\in \RRR_Q$ with entries as above the
arc in the Farey tessellation connecting $\frac{p}{q}$ and $\frac{p'}{q'}$, and conclude
that
\begin{equation*}
\# \RRR_Q =2\# \FF_Q-3 =\frac{Q^2}{\zeta (2)}+O(Q\ln Q).
\end{equation*}

\section{The coincidence of the pair correlations of \texorpdfstring{$\Phi$}{Phi} and
\texorpdfstring{$\Psi$}{Psi}}\label{Section4}

In this section we show that the limiting pair correlations of the sets $\{ \Psi (\gamma)\}$ and
$\{ \Phi (\gamma)\}$ ordered by $\| \gamma\| \rightarrow\infty$ do coincide. The proof uses properties of the
Farey tessellation, via the correspondence between elements of $\RRR_Q $ and arcs in the Farey tessellation
defined at the end of Section \ref{Section3}.

For $\gamma=\left( \begin{smallmatrix} p' & p \\ q' & q
\end{smallmatrix}\right)\in \RRR_Q$,  set
$\gamma_-=\frac{p}{q}$, $\gamma_+ =\frac{p^\prime}{q^\prime}$. From \eqref{3.1},
\eqref{3.3}, and the inequalities $X_\gamma<Z_\gamma<Y_\gamma$, $2Y_\gamma>T_\gamma$ and
$\epsilon_\gamma < \frac{1}{T_\gamma}$,
we have:
\begin{equation}\label{4.1}
\Psi (\gamma)-\Phi (\gamma)=\frac{Z_\gamma}{Y_\gamma (\epsilon_\gamma^{-1}Y_\gamma-1)}\ll \frac{1}{\| \gamma\|^4},
\end{equation}
\begin{equation}\label{4.2}
\gamma_- < \Phi (\gamma) < \Psi (\gamma) < \gamma_+ .
\end{equation}

Denote by $\RR_Q^\Psi (\xi)$,
respectively $\RR_Q^\Phi (\xi)$, the number of pairs $(\gamma,\gamma^\prime)\in \wMM_Q^2$, $\gamma\neq \gamma^\prime$, such that
$0\leqslant \Psi (\gamma)-\Psi
(\gamma^\prime) \leqslant \frac{\xi}{Q^2}$, respectively $0\leqslant \Phi (\gamma)-\Phi
(\gamma^\prime) \leqslant \frac{\xi}{Q^2}$.
For fixed $\beta \in ( \frac{2}{3},1)$, consider also
\begin{equation*}
\NN^\Psi_{Q,\xi,\beta} := \# \big\{ (\gamma,\gamma^\prime)\in\wMM_Q^2 :
Q^2 \vert \Psi (\gamma)-\Psi (\gamma^\prime)\vert \leqslant \xi ,\ \|\gamma\|
\leqslant Q^{\beta} \big\} .
\end{equation*}
and the similarly defined $\NN^\Phi_{Q,\xi,\beta}$. The trivial inequality
\begin{equation*}
\RR_Q^\Phi (\xi) \leqslant 2\NN^{\Phi}_{Q,\xi,\beta}  +
\# \big\{ (\gamma,\gamma^\prime)\in\wMM_Q^2 :\gamma\neq \gamma^\prime, Q^2 \vert \Phi (\gamma)-\Phi (\gamma^\prime)\vert \leqslant
\xi,\ \| \gamma\|,\|\gamma^\prime\| \geqslant Q^\beta \big\}
\end{equation*}
and the estimate in \eqref{4.1} show that there exists a universal constant $\kappa >0$
such that
\begin{equation*}
\RR_Q^\Phi (\xi)  \leqslant 2\NN^\Phi_{Q,\xi,\beta}  +\# \big\{
(\gamma,\gamma^\prime)\in\wMM_Q^2 : \gamma\neq\gamma^\prime,\
-2\kappa Q^{-4\beta} \leqslant \Psi (\gamma) -\Psi (\gamma^\prime) \leqslant
\xi Q^{-2} +2\kappa Q^{-4\beta} \big\},
\end{equation*}
showing that
\begin{equation}\label{4.3}
\RR_Q^\Phi (\xi) \leqslant 2\NN^\Phi_{Q,\xi,\beta} +\RR_Q^\Psi ( 2\kappa Q^{2-4\beta}) +
\RR_Q^\Psi ( \xi+2\kappa Q^{2-4\beta} ) .
\end{equation}
In a similar way we show that
\begin{equation}\label{4.4}
\RR_Q^\Psi (\xi) \leqslant 2\NN^\Psi_{Q,\xi,\beta} +\RR_Q^\Phi ( 2\kappa Q^{2-4\beta}) +
\RR_Q^\Phi ( \xi+2\kappa Q^{2-4\beta} ) .
\end{equation}

We first prove that $\NN^\Phi_{Q,\xi,\beta}$ and $\NN^\Psi_{Q,\xi,\beta}$ are much smaller
than $Q^2$. For this goal and for latter use, it is important to divide pairs $(\gamma, \gamma')\in \RRR_Q^2$ in
three cases, depending on the relative position of their associated arcs in the Farey tessellation (it is well
known that two arcs in the Farey tessellation are nonintersecting):
\begin{enumerate}
\item[(i)] The arcs  corresponding to $\gamma$ and $\gamma^\prime$ are \emph{exterior}, i.e. $\gamma_+
\leqslant \gamma^\prime_-$
or $\gamma^\prime_+ \leqslant \gamma_-$.
\item[(ii)] $\gamma^\prime \lessapprox \gamma$, i.e. $\gamma_- \leqslant \gamma^\prime_- <
\gamma^\prime_+ \leqslant \gamma_+$.
\item[(iii)] $\gamma \lessapprox \gamma^\prime$, i.e. $\gamma^\prime_- \leqslant \gamma_-
< \gamma_+ \leqslant \gamma^\prime_+$.
\end{enumerate}

\begin{proposition}\label{P3}
$\NN^\Phi_{Q,\xi,\beta} \ll Q^{1+\beta}\ln Q$ and $\NN^\Psi_{Q,\xi,\beta} \ll
Q^{1+\beta}\ln Q$.
\end{proposition}

\begin{proof}
$\NN^\Phi_{Q,\xi,\beta}$ and $\NN^\Psi_{Q,\xi,\beta}$ are increasing as an
effect of enlarging $\wMM_Q$ to $\RRR_Q$,
so for this proof we will replace $\wMM_Q$ by $\RRR_Q$. We only consider
$\NN^\Phi_{Q,\xi,\beta}$ here. The proof for
the bound on  $\NN^\Psi_{Q,\xi,\beta}$ is identical. Both rely on \eqref{4.1} and \eqref{4.2}.

Denote $K=[ \xi ]+1$. Upon \eqref{4.2}
and $\vert r^\prime -r\vert \geqslant \frac{1}{Q^2}$, $\forall r,r^\prime \in \FF_Q$,
$r\neq r^\prime$, it follows that
if $\gamma_+ \leqslant \gamma^\prime_-$ and
$\vert \Phi (\gamma^\prime)-\Phi (\gamma)\vert \leqslant \frac{\xi}{Q^2}$, then $\# (\FF_Q
\cap [\gamma_+, \gamma^\prime_- ]) \leqslant K+1$. In
particular $\gamma^\prime_- =\gamma_+$ when $0<\xi <1$.

We now consider the three cases enumerated before the statement of the proposition.

(i) The arcs corresponding to $\gamma$ and $\gamma^\prime$ are exterior. Without loss of
generality assume $\gamma_+ \leqslant \gamma_-^\prime$.
If $i$ is such that $\gamma_+=\gamma_i$, the $i^{\operatorname{th}}$ element of $\FF_Q$,
then
$\gamma^\prime_-=\gamma_{i+r}=\frac{p_{i+r}}{q_{i+r}}$ for some $r$ with $0\leqslant r<K$. The
equality $p^\prime_+ q^\prime_- - p^\prime_- q^\prime_+ =1$
shows that if $\gamma^\prime_-=\frac{p^\prime_-}{q^\prime_-}$ is fixed, then $q^\prime_+$
(and therefore $\gamma^\prime_+=\frac{p^\prime_+}{q^\prime_+}$)
is uniquely determined in intervals of length $\leqslant q^\prime_-$. Since $q^\prime_\pm
\leqslant Q$, it follows that the number of choices for
$q^\prime_+$ is actually $\leqslant \frac{Q}{q^\prime_-} +1=\frac{Q}{q_{i+r}}+1$.

When $0<\xi <1$ one must have $\gamma_-^\prime =\gamma_+$. Knowing $q_-$ and $q_+$ would
uniquely
determine the matrix $\gamma$. Then there will be at most $\frac{Q}{q_+}+1$ choices for
$\gamma^\prime$, so the total contribution
of this case to $\NN^\Phi_{Q,\xi,\beta}$ is
\begin{equation*}
\leqslant \sum_{1\leqslant q_-\leqslant Q^\beta} \sum_{1\leqslant q_+ \leqslant Q^\beta}
\bigg( \frac{Q}{q_+} +1\bigg) \ll Q^{1+\beta} \ln Q .
\end{equation*}
When $\xi \geqslant 1$ denote by $q_i,q_{i+1},\ldots,q_{i+K}$ the denominators of
$\gamma_i,\gamma_{i+1},\ldots,\gamma_{i+K}$.
Since $q_i<Q^\beta$, we have $\gamma_{i+K}-\gamma_i \leqslant \frac{K}{Q}\leqslant \frac{1}{Q^\beta} \leqslant \frac{1}{q_i}\leqslant 1-\gamma_i$,
showing that $i+K < \# \FF_Q$ so long as $Q\gg_\xi 1$.
As noticed in \cite{HT}, $q_{j+2}=[ \frac{Q+q_j}{q_{j+1}}] q_{j+1}-q_j$. As in
\cite{BCZ} consider
$\kappa (x,y) :=[ \frac{1+x}{y}]$ and $\TT_k =\{ (x,y)\in (0,1]^2: x+y >1, \kappa
(x,y)=k\}$.
Let $Q$ large enough so that $\delta_0:=Q^{\beta -1} <\frac{1}{2K+3}$. Then $\frac{q_i}{Q}
<\delta_0$ and it is plain
(cf. also \cite{BCZ}) that $\frac{q_{i+1}}{Q} >1-\delta_0$, $\kappa (\frac{q_i}{Q},\frac{q_{i+1}}{Q})=1$, and $\kappa (
\frac{q_{i+1}}{Q},\frac{q_{i+2}}{Q}) = \cdots
=\kappa ( \frac{q_{i+K}}{Q},\frac{q_{i+K+1}}{Q}) =2$ because
$q_{i+1},q_{i+2},\ldots,q_{i+K+1}$ must form an arithmetic progression. Hence $(
\frac{q_i}{Q},\frac{q_{i+1}}{Q})\in\TT_1$
and $( \frac{q_{i+1}}{Q},\frac{q_{i+2}}{Q}) ,\ldots ,(
\frac{q_{i+K}}{Q},\frac{q_{i+K+1}}{Q})\in \TT_2$,
showing in particular that $\min\{ q_{i+1},\ldots,q_{i+K}\} >\frac{Q}{3}$. Therefore $\max
\{ \frac{Q}{q_{i+1}},\ldots,\frac{Q}{q_{i+K}}\}<3$
and the contribution of this case to $\NN^\Phi_{Q,\xi,\beta}$ is
\begin{equation*}
\leqslant \sum_{1\leqslant q_- \leqslant Q^\beta} \sum_{1\leqslant q_+ \leqslant
Q^\beta} 4K \ll_\xi Q^{2\beta} .
\end{equation*}

\begin{figure}[ht]
\begin{center}
\unitlength 0.4mm
\begin{picture}(0,110)(0,-12)

\path(-130,0)(130,0)
\dottedline{2}(-77,0)(-77,100)

%\put(200,0){\arc{200}{3.14}{-2}}
\put(0,0){\arc{200}{3.14}{0}}

\put(50,0){\arc{100}{3.14}{0}}
\put(-50,0){\arc{100}{3.14}{0}}
%\put(150,0){\arc{100}{3.14}{-1.5}}
%\put(-150,0){\arc{100}{-1.6}{0}}

\put(16.666,0){\arc{33.333}{3.14}{0}}
\put(66.666,0){\arc{66.666}{3.14}{0}}
\put(-16.666,0){\arc{33.333}{3.14}{0}}
\put(-66.666,0){\arc{66.666}{3.14}{0}}

%\put(133.333,0){\arc{66.666}{3.14}{-1.3}}
%\put(125,0){\arc{50}{3.14}{-1}}
%\put(-133.333,0){\arc{66.666}{-1.75}{0}}
%\put(-125,0){\arc{50}{-1.9}{0}}

\put(41.666,0){\arc{16.666}{3.14}{0}}
\put(75,0){\arc{50}{3.14}{0}}
\put(10,0){\arc{20}{3.14}{0}}
\put(26.666,0){\arc{13.333}{3.14}{0}}

\put(-41.666,0){\arc{16.666}{3.14}{0}}
\put(-75,0){\arc{50}{3.14}{0}}
\put(-10,0){\arc{20}{3.14}{0}}
\put(-26.666,0){\arc{13.333}{3.14}{0}}

\put(-100,-12){\makebox(0,0){$\frac{p}{q}$}}
\put(100,-12){\makebox(0,0){$\frac{p^\prime}{q^\prime}$}}
\put(0,-12){\makebox(0,0){$\frac{p+p^\prime}{q+q^\prime}$}}
\put(-32,-12){\makebox(0,0){{\small $\frac{2p+p^\prime}{2q+q^\prime}$}}}
\put(33.33,-12){\makebox(0,0){{\small $\frac{p+2p^\prime}{q+2q^\prime}$}}}
\put(-52,-12){\makebox(0,0){{\small $\frac{3p+p^\prime}{3q+q^\prime}$}}}
%\put(50,-12){\makebox(0,0){{\small $\frac{3}{4}$}}}
\put(-55,92){\makebox(0,0){$\gamma$}}
\put(-77,-12){\makebox(0,0){$\Phi (\gamma)$}}

\end{picture}
\end{center}
\caption{The Farey tessellation.} \label{Figure3}
\end{figure}
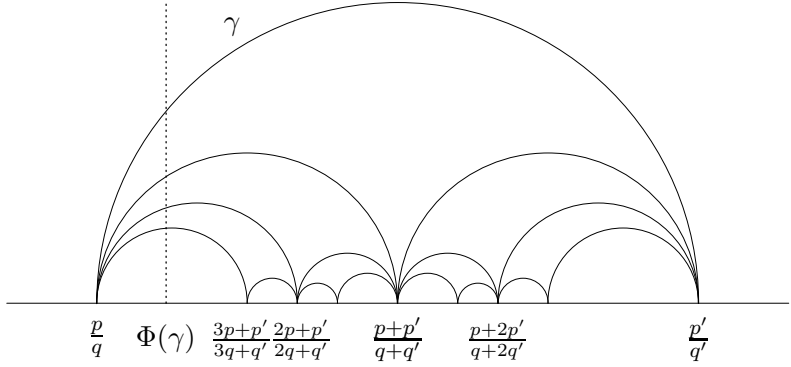

(ii) $\gamma^\prime \lessapprox \gamma$. Let $i$ be the unique index for which $\gamma_i <\Phi (\gamma)
<\gamma_{i+1}$ with
$\gamma_i < \gamma_{i+1}$ successive elements in $\FF_Q$.
Since $\vert \Phi(\gamma^\prime)-\Phi(\gamma)\vert \leqslant \frac{\xi}{Q^2}$,
either $\gamma^\prime_- < \Phi (\gamma) < \gamma^\prime_+$ or there exists $0\leqslant r
\leqslant K$ with $\gamma^\prime_+ =\gamma_{i-r}$ or with
$\gamma^\prime_- =\gamma_{i+r}$. In both situations the arc corresponding to the matrix $\gamma^\prime$ will cross at
least one of the vertical lines above
$\gamma_{i-K}, \ldots ,\gamma_i, \gamma_{i+1},\ldots ,\gamma_{i+K}$. A glance at the Farey
tessellation provides an upper
bound for this number $N_{\gamma,K}$ of arcs $\gamma^\prime\in\RRR_Q$.
Actually one sees that the set $\CC_{\gamma,L}$ consisting of $2+2^2+\cdots+2^L$ arcs
obtained from $\gamma$ by
iterating the mediant construction $L=[ \frac{Q}{\min\{ q_-,q_+\}}]+1$ times ($\gamma$
is not enclosed in $\CC_{\gamma,L}$)
contains the set $\{ \gamma^\prime\in \RRR_Q: \gamma^\prime \lessapprox \gamma,
\gamma^\prime \neq \gamma\}$.
The former set contains at most $L$ arcs that are intersected by each vertical direction,
and so $N_{\gamma,K} \leqslant (2K+1)L$.
Therefore, the contribution of this case to $\NN^\Phi_{Q,\xi,\beta}$ is (first choose
$\gamma$, then $\gamma^\prime$)
\begin{equation*}
\leqslant \sum_{1\leqslant q\leqslant Q^\beta} \sum_{1\leqslant q^\prime \leqslant
Q^\beta}
(2K+1) \bigg( \frac{Q}{\min\{ q,q^\prime\}} +1 \bigg) \ll_\xi Q^{1+\beta} \ln Q .
\end{equation*}

(iii) $\gamma \lessapprox \gamma^\prime$. We necessarily have $\gamma =\gamma^\prime
M$, with $M\in {\mathfrak S}$. In particular this yields
$\gamma_\pm^\prime \in \FF_{Q^\beta}$. Considering the sub-tessellation defined only by
arcs connecting
points from $\FF_{Q^\beta}$, one sees that the number of arcs intersected by a vertical
line $x=\alpha$
with $\gamma_-=\frac{p}{q} <\alpha <\gamma_+ =\frac{p^\prime}{q^\prime}$, $\gamma
=(\gamma_-,\gamma_+)\in \FF_{Q^\beta}$
is equal to $s(q,q^\prime)$, the sum of digits in the continued fraction expansion of
$\frac{q}{q^\prime}<1$ when $q<q^\prime$, and respectively to $s(q^\prime,q)$ when
$q^\prime <q$.
A result from \cite{YK} yields in particular that
\begin{equation*}
\sum_{0<q<q^\prime\leqslant Q^\beta} s(q,q^\prime) \ll Q^{2\beta} \ln^2 Q,
\end{equation*}
and therefore
\begin{equation*}
\# \{ (\gamma,\gamma^\prime) \in \RRR^2_{Q^\beta} : \gamma \lessapprox \gamma^\prime \}
\leqslant 1+2 \sum_{0<q<q^\prime \leqslant Q^\beta} s(q,q^\prime) \ll Q^{2\beta} \ln^2 Q.
\end{equation*}
This completes the proof of the proposition.
\end{proof}

Proposition \ref{P3} and inequalities \eqref{4.3} and \eqref{4.4} provide

\begin{cor}\label{C4}
For each $\beta \in ( \frac{2}{3},1)$,
\begin{equation*}
\RR_Q^\Psi (\xi) = \RR_Q^\Phi \big( \xi+O_\xi (Q^{2-3\beta}) \big)
+\RR_Q^\Phi \big( O_\xi (Q^{2-3\beta})\big)  +O_\xi (Q^{1+\beta}\ln Q) .
\end{equation*}
\end{cor}

\section{A decomposition of the pair correlation of
\texorpdfstring{$\{\Phi(\gamma)\}$}{Phi[gamma]}} \label{Section5}

To estimate $\RR_Q^\Phi(\xi)$, recall the correspondence between elements of $\RRR_Q$ and arcs in the Farey
tessellation from the end of Section \ref{Section3}. We consider the
following two possibilities for the arcs associated to a pair $(\gamma,
\gamma')\in \wMM_Q^2$:
\begin{itemize}
\item[(i)] One of the arcs corresponding to $\gamma$ and $\gamma^\prime$ contains the other.
\item[(ii)] The arcs corresponding to $\gamma$ and $\gamma^\prime$ are exterior (possibly tangent).
\end{itemize}
Denoting by $R_Q^\Cap( \xi), R_Q^{\cap\, \cap} (\xi) $ the number of pairs in each case we have
$$\RR_Q^\Phi(\xi)=R_Q^\Cap( \xi)+  R_Q^{\cap\, \cap} (\xi).
$$

\subsection{One of the arcs contains the other}\label{sec5.1}
In this case we have either $\gamma=\gamma^\prime M$ or $\gamma^\prime =\gamma M$ with $M\in {\mathfrak S}$
(see also Figure \ref{Figure4}). For each $M\in \Gamma$ define
\begin{equation}\label{5.1}
\Xi_M (x,y) = \frac{xy(Y_M-X_M)+(x^2-y^2)Z_M}{(x^2+y^2)(x^2 X_M+y^2 Y_M +2xy Z_M)},
\end{equation}
where $X_M,Y_M,Z_M$ are defined in \eqref{3.1}. A direct calculation leads
for $\gamma=\left( \begin{smallmatrix} p^\prime & p \\ q^\prime & q \end{smallmatrix}\right)$ to
\begin{equation}\label{5.2}
\Phi (\gamma)-\Phi (\gamma M)=\Xi_M (q^\prime ,q).
\end{equation}
Tho remarks are in order now. First notice that
$X_M \neq Y_M$ for any $M\in {\mathfrak S}$ because of \eqref{3.2} and $X_M,Y_M,Z_M\geqslant 1$.
Secondly, we also have
\begin{equation}\label{5.3}
\Phi(\gamma)\neq \Phi (\gamma M).
\end{equation}
Suppose ad absurdum that $\Phi(\gamma)=\Phi(\gamma M)$. Then \eqref{5.2} and \eqref{5.1} yield
$\frac{2Z_M}{Y_M -X_M} =\frac{2qq^\prime}{q^2-q^{\prime 2}}$, that is
$\tan\theta_M =\tan 2\theta$, where $\theta =\tan^{-1}(\frac{q^\prime}{q})\in (0,\pi)$ and
$\theta_M \in (0,\pi)$ because $Z_M >0$. This gives
\[
\frac{X_M -\epsilon_{M}}{Z_M} =\tan \bigg( \frac{\theta_M}{2}\bigg)=\tan\theta \in \Q,
\]
hence $\sqrt{(X_M+Y_M)^2-4}=X_M+Y_M -2\epsilon_M \in \Q$, which is not possible because $X_M+Y_M\geqslant 3$.

From \eqref{5.2} and \eqref{5.3} we now infer

\begin{lemma}\label{L5}
The number of pairs $(\gamma,\gamma^\prime) \in \wMM_Q^2$,
$\gamma\neq\gamma^\prime$ ,   with
$0\leqslant \Phi (\gamma)-\Phi (\gamma^\prime) \leqslant \frac{\xi}{Q^2}$ and
$\gamma \lessapprox \gamma^\prime$ or $\gamma^\prime \lessapprox \gamma$ (with the notation introduced before Proposition
\ref{P3}) is given by
\begin{equation*}
R_Q^\Cap( \xi) = \# \bigg\{ (\gamma,\gamma M)\in\wMM_Q^2 :\gamma=\left(
\begin{matrix} p^\prime & p \\ q^\prime & q
\end{matrix}\right), \ M\in{\mathfrak S} , \
\vert \Xi_M (q^\prime ,q)\vert \leqslant \frac{\xi}{Q^2} \bigg\} .
\end{equation*}
\end{lemma}

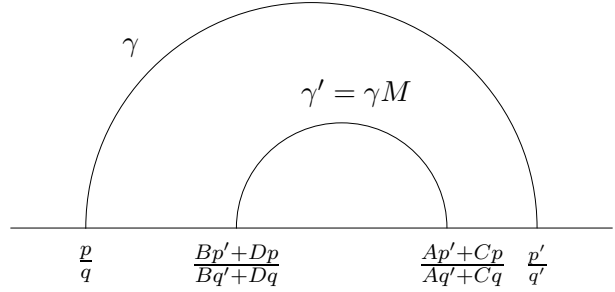
\begin{figure}[ht]
\begin{center}
\unitlength 0.4mm
\begin{picture}(0,75)(10,-5)

\path(-100,0)(100,0)
\put(0,0){\arc{150}{3.14}{0}}
\put(10,0){\arc{70}{3.14}{0}}

\put(-60,60){\makebox(0,0){{$\gamma$}}}
\put(15,45){\makebox(0,0){{$\gamma^\prime =\gamma M$}}}

\put(-75,-12){\makebox(0,0){{$\frac{p}{q}$}}}
\put(75,-12){\makebox(0,0){{\small $\frac{p^\prime}{q^\prime}$}}}
\put(-25,-12){\makebox(0,0){{$\frac{Bp^\prime+Dp}{Bq^\prime+Dq}$}}}
\put(50,-12){\makebox(0,0){{$\frac{Ap^\prime+Cp}{Aq^\prime+Cq}$}}}
\end{picture}
\end{center}
\caption{\small The case $\gamma^\prime \lessapprox \gamma$} \label{Figure4}
\end{figure}

\subsection{Exterior arcs}\label{sec5.2}
In this case we have $\gamma,\gamma^\prime
\in  \wMM_Q$, $\gamma^\prime_- \geqslant \gamma_+$. Let $\ell\geqslant 0$ be the number of Farey arcs in $\FF_Q$
connecting the arcs corresponding to $\gamma,\gamma^\prime $ (see Figure \ref{Figure5}). In other words, writing $\gamma=\Big(
\begin{smallmatrix} p' & p \\ q' & q
\end{smallmatrix}\Big), \gamma'=\Big( \begin{smallmatrix} p_{\ell+1} & p_{\ell} \\ q_{\ell+1} & q_{\ell}
\end{smallmatrix}\Big)$, we have that $\frac{p_0}{q_0}:=\frac{p'}{q'},
\frac{p_1}{q_1},\ldots,\frac{p_\ell}{q_\ell}$  are consecutive elements in $\FF_Q$. Setting also $
\frac{p_{-1}}{q_{-1}}:=\frac{p}{q}$, it follows that $q_{i}=k_i
q_{i-1} -q_{i-2}$, where $k_i\in \N$, $i=1,\ldots,\ell$, and $k_i=\big[ \frac{Q+q_{i-2}}{q_{i-1}}\big]$ for
$2\leqslant i\leqslant \ell$.

The fractions $\frac{p_\ell}{q_\ell},
\frac{p_{\ell+1}}{q_{\ell+1}}$ are not necessarily consecutive in $\FF_Q$, but we have $q_{\ell+1}=Kq_\ell-q_{\ell
-1}$, $K\leqslant k_{\ell+1}
=\big[ \frac{Q+q_\ell}{q_{\ell+1}}\big]$. It follows that $\gamma^\prime = \gamma M$ with
$M=\Big( \begin{smallmatrix} k_1 & 1 \\ -1 & 0 \end{smallmatrix}\Big)
\cdots \Big( \begin{smallmatrix} k_\ell & 1 \\ -1 & 0 \end{smallmatrix}\Big)\Big( \begin{smallmatrix} K & 1 \\
-1 & 0 \end{smallmatrix}\Big).$

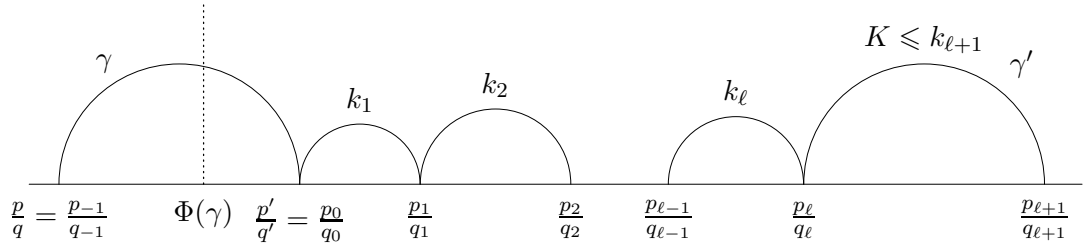
\begin{figure}[ht]
\begin{center}
\unitlength 0.4mm
\begin{picture}(0,80)(30,-12)

\path(-150,0)(200,0)
\dottedline{2}(-92,0)(-92,60)

\put(-100,0){\arc{80}{3.14}{0}}
\put(-40,0){\arc{40}{3.14}{0}}
\put(5,0){\arc{50}{3.14}{0}}

\put(85,0){\arc{45}{3.14}{0}}
\put(147.5,0){\arc{80}{3.14}{0}}

\put(-140,-12){\makebox(0,0){$\frac{p}{q}=\frac{p_{-1}}{q_{-1}}$}}
\put(-20,-12){\makebox(0,0){{$\frac{p_1}{q_1}$}}}
\put(30,-12){\makebox(0,0){{$\frac{p_2}{q_2}$}}}
\put(-60,-12){\makebox(0,0){{$\frac{p^\prime}{q^\prime}=\frac{p_0}{q_0}$}}}
\put(62.5,-12){\makebox(0,0){$\frac{p_{\ell-1}}{q_{\ell-1}}$}}
\put(107.5,-12){\makebox(0,0){$\frac{p_\ell}{q_\ell}$}}
\put(187.5,-12){\makebox(0,0){$\frac{p_{\ell+1}}{q_{\ell+1}}$}}
\put(-125,40){\makebox(0,0){$\gamma$}}
\put(180,40){\makebox(0,0){$\gamma^\prime$}}
\put(147.5,48){\makebox(0,0){$K\leqslant k_{\ell +1}$}}
\put(85,30.5){\makebox(0,0){$k_\ell$}}
\put(5,33){\makebox(0,0){$k_2$}}
\put(-40,28){\makebox(0,0){$k_1$}}
\put(-92,-10){\makebox(0,0){$\Phi (\gamma)$}}

\end{picture}
\end{center}
\caption{\small The case where $\gamma$ and $\gamma^\prime$ are exterior} \label{Figure5}
\end{figure}

We have $\ell <\xi$ because
\begin{equation*}
\Phi (\gamma^\prime)-\Phi (\gamma) > \sum_{i=1}^\ell \frac{1}{q_{i-1} q_i} \geqslant \frac{\ell}{Q^2} .
\end{equation*}
It is also plain to see that
\begin{equation}\label{5.4}
 \frac{p^\prime}{q^\prime}-\Phi (\gamma) =\frac{q}{q^\prime (q^2+q^{\prime 2})} ,\quad
 \Phi (\gamma^\prime)-\frac{p_\ell}{q_\ell} =\frac{q_{\ell+1}}{q_\ell (q_\ell^2+q_{\ell+1}^2)} .
\end{equation}
The last equality in \eqref{5.4} and $q_\ell^2+q_{\ell+1}^2 \leqslant Q^2$ yield for $\ell \geqslant 1$
\begin{equation*}
\begin{split}
\frac{\xi}{Q^2} & \geqslant \Phi (\gamma^\prime)-\Phi (\gamma) \geqslant  \frac{1}{q_{\ell-1}q_\ell} +
\frac{q_{\ell+1}}{q_\ell (q_\ell^2+q_{\ell+1}^2)} \\ &
\geqslant \frac{1}{q_{\ell-1}q_\ell}+\frac{Kq_\ell -q_{\ell-1}}{q_\ell Q^2}
=\frac{K}{Q^2}+\frac{Q^2-q_{\ell-1}^2}{q_{\ell-1}q_\ell Q^2}\geqslant \frac{K}{Q^2},
\end{split}
\end{equation*}
while if $\ell=0$ we have
$$\Phi (\gamma^\prime)-\Phi (\gamma)= \frac{K(q^{\prime 2}+qq_1)}{(q^2+
q^{\prime 2})(q^{\prime 2}+q_1^{2})} \geqslant \frac{K}{Q^2},$$
showing that $K <\xi$. Notice also that \eqref{5.4} yields
\begin{equation*}
\Phi(\gamma^\prime)-\Phi (\gamma) = \frac{q}{q^\prime (q^2+q^{\prime 2})}
+\sum_{i=1}^\ell \frac{1}{q_{i-1}q_i} + \frac{q_{\ell+1}}{q_\ell (q_\ell^2+q_{\ell+1}^2)} .
\end{equation*}

Let $\TT=\{ (x,y)\in (0,1]^2 : x+y >1\}$ and consider the map
\begin{equation*}
T:(0,1]^2 \rightarrow \TT,\quad T(x,y)=\bigg( y,\bigg[ \frac{1+x}{y}\bigg] y-x\bigg),
\end{equation*}
whose restriction to $\TT$ is bijective and area-preserving \cite{BCZ}.
Consider the iterates $T^i=(L_{i-1},L_i)$ and the functions $K_i = [ \frac{1+L_{i-2}}{L_{i-1}}]$ if $i=1,\ldots,\ell$,
$K_{\ell+1}=K$, and $L_{\ell+1}=KL_\ell-L_{\ell-1}$. One has
\begin{equation}\label{5.5}
\begin{split}
& L_{-1} (x,y) =x , \quad L_0 (x,y)=y,\quad (x,y)\in (0,1]^2, \\
& 0 < L_i (x,y)\leqslant 1,\ \  i\geqslant 0, \quad L_{i-1}(x,y)+L_i (x,y) >1 ,\ \  i=1,\ldots,\ell,\quad (x,y)\in \TT, \\
& L_i (x,y) = K_i (x,y) L_{i-1}(x,y)-L_{i-2}(x,y),\quad i=1,\ldots,\ell+1,\quad  (x,y)\in \TT, \\
& (q_{i-1},q_i) =QT^i \bigg(\frac{q}{Q},\frac{q^\prime}{Q}\bigg) =\Bigg( QL_{i-1}\bigg( \frac{q}{Q},\frac{q^\prime}{Q}\bigg) ,
QL_i \bigg( \frac{q}{Q}, \frac{q^\prime}{Q}\bigg) \Bigg),\quad i=0,1,\ldots,\ell, \\
& q_{\ell+1}=Kq_\ell -q_{\ell -1} =Q \Bigg( K L_\ell \bigg( \frac{q}{Q},\frac{q^\prime}{Q}\bigg)
-L_{\ell-1} \bigg( \frac{q}{Q},\frac{q^\prime}{Q}\bigg)\Bigg) .
\end{split}
\end{equation}
Define also the function
\begin{equation}\label{5.6}
\Upsilon_{\ell,K}:(0,1]^2 \rightarrow (0,\infty),\quad
\Upsilon_{\ell,K} =\frac{L_{-1}}{L_0 (L_{-1}^2+L_0^2)} +\sum_{i=1}^\ell \frac{1}{L_{i-1}L_i}
+\frac{L_{\ell+1}}{L_\ell \big( L_\ell^2+L_{\ell+1}^2\big)} .
\end{equation}

We proved the following statement.

\begin{lemma}\label{L6}
The number $R_Q^{\cap\, \cap} (\xi)$ of pairs
$(\gamma,\gamma^\prime)$ of exterior (possibly tangent) arcs in $\wMM_Q$ for which
$0< \Phi (\gamma^\prime)-\Phi (\gamma) \leqslant \frac{\xi}{Q^2}$
is given by
\begin{equation}\label{5.7}
R_Q^{\cap\,\cap} (\xi) = \sum\limits_{\substack{\ell\in [0,\xi) \\ K \in [1,\xi)}} \#
\left\{ \left( \begin{matrix} p^\prime & p \\ q^\prime & q \end{matrix}\right) :
\begin{matrix}
0\leqslant p\leqslant q,\ 0\leqslant p^\prime \leqslant q^\prime,\ p^\prime q-pq^\prime =1 \\
p^2+p^{\prime 2}+q^2+q^{\prime 2} \leqslant Q^2 , \  0<Kq_\ell -q_{\ell -1} \leqslant Q \\
p_\ell^2+q_\ell^2 +(K p_\ell-p_{\ell-1})^2 +(Kq_\ell -q_{\ell-1})^2 \leqslant Q^2 \\
\Upsilon_{\ell,K} \big( \frac{q}{Q},\frac{q^\prime}{Q}\big) \leqslant \xi
\end{matrix}\right\},
\end{equation}
where the sums are over integers in the given intervals, and $q_{-1}=q$, $q_0=q'$.
\end{lemma}

\section{A lattice point estimate}\label{Sect6}

\begin{lemma}\label{L7}
Suppose that $\Omega$ is a region in $\R^2$ of area $A(\Omega)$ and rectifiable boundary of length $\ell (\partial \Omega)$.
For every integer $r$ with $(r,q)=1$ and $1\leqslant L\leqslant q$
\begin{equation*}
\NN_{\Omega,q,r} := \# \big\{ (a,b)\in \Omega \cap \Z^2 : ab\equiv r \hspace{-6pt}\pmod{q}\big\} =
\frac{\varphi (q)}{q^2}  A(\Omega) +\EE_{\Omega,L,q},
\end{equation*}
where, for each $\veps >0$,
\begin{equation*}
\EE_{\Omega,L,q} \ll_\veps \frac{q^{1/2+\veps} A(\Omega)}{L^2} + \bigg( 1+\frac{\ell (\partial\Omega)}{L}\bigg)
\bigg( \frac{L^2}{q} +q^{1/2+\veps}\bigg) .
\end{equation*}
\end{lemma}

\begin{proof}
Replacing $\Z^2$ by $L\Z^2$ in the estimate (for a proof see \cite[Thm. 5.9]{Nar})
\begin{equation*}
\big\{ (m,n)\in\Z^2 : (m,m+1) \times (n,n+1) \cap \partial \Omega \neq \emptyset \big\}
\ll 1+\ell (\partial \Omega) ,
\end{equation*}
we find that the number of squares $S_{m,n}=[Lm,L(m+1)] \times [Ln,L(n+1)]$ with
$\mathring{S}_{m,n} \cap \partial \Omega \neq \emptyset$ is
$\ll 1+\frac{1}{L} \ell (\partial \Omega)$. Therefore
\begin{equation*}
\# \big\{ (m,n)\in \Z^2 : (Lm,L(m+1)) \times (Ln,L(n+1)) \subseteq \Omega \big\}
=\frac{A(\Omega)}{L^2} +O\bigg( 1+\frac{\ell(\partial \Omega)}{L}\bigg) .
\end{equation*}
Weil's estimates on Kloosterman sums \cite{We} extended to composite moduli
in \cite{Hoo} and \cite{Est} show that each such square contains $\frac{\varphi (q)}{q^2}  L^2 +O_\veps (q^{1/2 +\veps})$
pairs of integers $(a,b)$ with $ab \equiv r \hspace{-3pt}\pmod{q}$ (see, e.g. \cite[Lemma 1.7]{BCZ0} for details).
Combining these two estimates we find
\begin{equation*}
\NN_{\Omega,q,r} = \bigg( \frac{A(\Omega)}{L^2} +O \Big( 1+\frac{\ell (\partial \Omega)}{L}\Big) \bigg)
\bigg( \frac{\varphi (q)}{q^2}  L^2 +O(q^{1/2+\veps})\bigg) =
\frac{\varphi (q)}{q^2} A(\Omega) +\EE_{\Omega,q,L},
\end{equation*}
as desired.
\end{proof}

\begin{cor}\label{C8}
$\displaystyle \quad \# \wMM_Q =\frac{3Q^2}{8}  +O_\veps ( Q^{11/6+\veps}) .$
\end{cor}

\begin{proof}
Note first that one can substitute $\frac{pq^\prime}{q}$ for
$p^\prime =\frac{1+pq^\prime}{q}$ in the definition of $\wMM_Q$, replacing the inequality
$\| \gamma\|^2 \leqslant Q^2$ by $(q^2+q^{\prime 2})(q^2+p^2) \leqslant Q^2 q^2$ without altering the error term.
Applying Lemma \ref{L7} to $\Omega_q =\{ (u,v)\in [0,q]\times [0,Q] : (q^2+u^2)(q^2+v^2)\leqslant Q^2 q^2 \}$ and
$L=q^{5/6}$, and using $A(\Omega_q)\leqslant Qq$ and $\ell (\Omega_q) \leqslant 2(Q+q)\leqslant 4Q$, we infer
\begin{equation*}
\# \wMM_Q = \sum_{q=1}^Q \frac{\varphi (q)}{q} \cdot \frac{A(\Omega_q)}{q} +O_\veps ( Q^{11/6+\veps}) .
\end{equation*}
Standard M\" obius summation (see, e.g., \cite[Lemma 2.3]{BCZ0}) applied to the decreasing function $h(q)=\frac{1}{q} A(\Omega_q)$ with $\| h\|_\infty \leqslant Q$
and the change of variable $(q,u,v)=(Qx,Qxy,Qz)$ further yield
\begin{equation*}
\# \wMM_Q =\frac{Q^2}{\zeta (2)} \operatorname{Vol} (S)+O_\veps ( Q^{11/6+\veps}) ,
\end{equation*}
where
\begin{equation*}
S=\{ (x,y,z)\in [0,1]^3: (1+y^2)(x^2+z^2) \leqslant 1 \} .
\end{equation*}
The substitution $y=\tan \theta$ yields
\begin{equation*}
\operatorname{Vol} (S) = \int_0^{\pi/4} \frac{d\theta}{\cos^2\theta} A \big( \{ (x,z)\in [0,1]^2 : x^2+z^2 \leqslant \cos^2 \theta \}\big)
=\frac{\pi^2}{16} ,
\end{equation*}
completing the proof of the corollary.
\end{proof}

The error bound in Corollary \ref{C8} can be improved using spectral methods (see Corollary 12.2 in
Iwaniec's book \cite{Iw}). We have given the proof since it is the
prototype of applying Lemma \ref{L7} to the counting problems of the next section.

\section{Pair correlation
of \texorpdfstring{$\{\Phi(\gamma)\}$}{Phi[gamma]}}\label{Section7}
The main result of this section is Theorem \ref{T2}, where we obtain explicit formulas for the pair correlation of
the quantities $\{\Phi(\gamma)\}$ in terms of volumes of three dimensional bodies. The discussion is divided in
two  cases, as in Sec. \ref{Section5}.

\subsection{One of the arcs contains the other}

The formula for $R_Q^\Cap$ in Lemma \ref{L5} provides
\begin{equation}\label{7.1}
R_Q^\Cap (\xi) = \sum\limits_{M\in {\mathfrak S}} \NN_{M,Q}(\xi) ,
\end{equation}
where $\NN_{M,Q} (\xi)$ denotes the number of matrices
$\gamma=\left( \begin{smallmatrix} p^\prime & p \\ q^\prime & q \end{smallmatrix}\right)$ for which
\begin{equation}\label{7.2}
0\leqslant p\leqslant q,\quad 0\leqslant p^\prime \leqslant q^\prime,\quad p^\prime q-pq^\prime =1, \quad
\displaystyle \vert \Xi_M (q^\prime ,q)\vert \leqslant \frac{\xi}{Q^2}  , \quad \| \gamma M\| \leqslant Q.
\end{equation}

The first goal is to replace in \eqref{7.2} the inequality $\| \gamma M\| \leqslant Q$ by a more tractable one.
Taking $\gamma=\Big( \begin{smallmatrix} p' & p \\ q' & q \end{smallmatrix}\Big)$ and substituting
$p=\frac{p^\prime q-1}{q^\prime}$ we write, using the notation \eqref{3.1}:
\begin{equation}\label{7.3}
 \| \gamma M\|^2  =\bigg( \frac{p^{\prime 2}}{q^{\prime 2}} +1\bigg) \big(q^{\prime 2} X_M +  q^2 Y_M + 2 q
q^\prime Z_M \big)
-\frac{ (p^\prime q+pq^\prime) Y_M +2p^\prime q^\prime Z_M }{q^{\prime 2}} .
\end{equation}

The quantity $\NN_{M,Q}(\xi)$ can be conveniently related to $\wNN_{M,Q} (\xi)$, the number of integer triples $(q^\prime,q,p^\prime)$ such that
\begin{equation}\label{7.4}
\begin{cases}
0< p^\prime \leqslant q^\prime \leqslant Q, \quad 0<q\leqslant Q ,\quad
p^\prime q \equiv 1 \hspace{-6pt} \pmod{q^\prime} ,\\  \displaystyle
\vert \Xi_M (q^\prime ,q)\vert \leqslant \frac{\xi}{Q^2}  , \quad
Y_{\gamma M}=q^{\prime 2} X_M +  q^2 Y_M + 2 q q^\prime Z_M \leqslant \frac{Q^2 q^{\prime 2}}{p^{\prime 2} +q^{\prime
2}}  .
\end{cases}
\end{equation}

We next prove that given $c_0 \in (\frac{1}{2},1)$, for all $M\in {\mathfrak S}$
and $Q\geqslant 1$ with $Y_M < X_M \leqslant Q^{2c_0}$ and all $\xi >0$,
\begin{equation}\label{7.5}
\wNN_{M,Q}(\xi) \leqslant \NN_{M,Q} (\xi) \leqslant \wNN_{M,Q(1+\sqrt{2} Q^{c_0-1})} \big(\xi(1+\sqrt{2} Q^{c_0-1})^2\big).
\end{equation}
For the first inequality, note that if the integral triple $(q^\prime,q,p^\prime)$ satisfies
\eqref{7.4} then by \eqref{7.3}, $\| \gamma M\|^2 \leqslant \frac{p^{\prime 2}+q^{\prime 2}}{q^{\prime 2}} Y_{\gamma M} \leqslant Q^2$,
 and thus if we define $p:=\frac{p^\prime q-1}{q^\prime}$ then \eqref{7.2} holds. For the second inequality take
 $\gamma$ as in \eqref{7.2}. Using \eqref{7.3} we then have
 $\frac{p^{\prime 2}+q^{\prime 2}}{q^{\prime 2}}
 Y_{\gamma M} \leqslant Q^2+\frac{(p^\prime q+pq^\prime)Y_M +2p^\prime q^\prime Z_M}{q^{\prime 2}}
 \leqslant Q^2+2qY_M+2Z_M $. Using also $Z_M \leqslant Q^{2c_0}$ and
 $qY_M =\sqrt{q^2 Y_M} \sqrt{Y_M} \leqslant \sqrt{Y_{\gamma M}} \sqrt{Y_M} \leqslant Q^{1+c_0}$, we conclude
 $\frac{p^{\prime 2}+q^{\prime 2}}{q^{\prime 2}} Y_{\gamma M} \leqslant Q^2 +2Q^{1+c_0} +2Q^{2c_0}\leqslant Q^2(1+\sqrt{2}Q^{c_0-1})^2$.
 Also $\vert \Xi_M (q^\prime,q)\vert \leqslant \frac{\xi}{Q^2}=\frac{\xi(1+\sqrt{2}Q^{c_0-1})^2}{Q^2 (1+\sqrt{2}Q^{c_0-1})^2}$.
 Hence $(q^\prime,q,p^\prime)$ satisfies \eqref{7.4} with the pair $(Q,\xi)$ replaced by
 $(Q+\sqrt{2}Q^{c_0},\xi (1+\sqrt{2} Q^{c_0-1})^2 )$. This proves \eqref{7.5}.

Next we show that $\NN_{M,Q} (\xi) =0$ when $\max\{ X_M,Y_M\} \geqslant Q^{2c_0}$  and
$Q$ is large enough.

\begin{lemma}\label{L9}
Let $c_0\in ( \frac{1}{2},1)$. There exists $Q_0(\xi)$ such that whenever $M\in {\mathfrak S}$,
$\max\{ X_M,Y_M\} \geqslant Q^{2c_0}$, and $Q\geqslant Q_0(\xi)$,
\begin{equation*}
\NN_{M,Q} (\xi)=\wNN_{M,Q}(\xi)=0.
\end{equation*}
\end{lemma}

\begin{proof}
We show there are no coprime positive integer lattice points $(q^\prime,q)$ for which
\begin{equation}\label{7.6}
\vert \Xi_M (q^\prime,q)\vert \leqslant \frac{\xi}{Q^2},\quad
Y_{\gamma M}=q^{\prime 2} X_M+q^2 Y_M +2qq^\prime Z_M \leqslant Q^2 .
\end{equation}
Noting from \eqref{7.3} that $Y_{\gamma M}\leqslant \| \gamma M\|^2$, this will
ensure that $\NN_{M,Q}(\xi)=0$. The equality $\wNN_{M,Q}(\xi)=0$ follows as well from \eqref{7.4}.

Suppose $(q^\prime,q)$ is as in \eqref{7.6}, write $q^\prime i+q=(q,q^\prime)=(r\cos\theta,r\sin\theta)$,
$\theta\in ( 0,\frac{\pi}{2})$, and consider $(X,Y,Z)=(X_M,Y_M,Z_M)$,
$T=\| M\|^2=X+Y$, $U_M =\coth d(i,Mi)=\frac{T}{\sqrt{T^2-4}}$.  Since
$\sin\theta_M =\frac{2Z}{\sqrt{T^2-4}}$ and
$\cos \theta_M = \frac{Y-X}{\sqrt{T^2-4}}$,
the inequalities in \eqref{7.6} can be described as
\begin{equation}\label{7.7}
\frac{1}{\xi}\cdot \frac{\vert \sin (\theta_M -2\theta)\vert}{U_M+\cos (\theta_M -2\theta)}
\leqslant \frac{r^2}{Q^2} \leqslant
\frac{2}{\big(U_M +\cos (\theta_M -2\theta)\big)\sqrt{T^2-4}} .
\end{equation}

Denoting $\delta_M=\frac{\theta_M}{2}-\theta $, from  the first and
last fraction in \eqref{7.7} we infer
$\vert \sin  2 \delta_M \vert \ll \frac{1}{T}$. Therefore $\delta_M$ is close to 0, or to
$\pm \frac{\pi}{2}$. When $\delta_M$ is close to $0$ we have
$|\tan \delta_M|\ll \vert \delta_M\vert \ll  \vert \sin 2\delta_M \vert \ll\frac{1}{T}$.
When $\delta_M$ is close to $\pm \frac{\pi}{2}$ we similarly have
$|\delta_M\mp \frac{\pi}{2}|\ll \frac{1}{T},$ which is seen to be impossible. Indeed,
\begin{equation*}
 \frac{|\tan\delta_M|}{1+\frac{U_M-1}{1+\cos 2\delta_M}}
=\frac{\vert \sin 2\delta_M \vert}{U_M+\cos 2\delta_M} \leqslant \xi
\end{equation*}
shows that it suffices to bound from above $\frac{U_M-1}{1+\cos 2\delta_M}$, which would imply
$|\tan\delta_M|\ll \xi$, thus contradicting $|\delta_M\mp \frac{\pi}{2}|\ll \frac{1}{T}$. Since $Z$ is a positive
integer, we have $ \sin \theta_M\gg \frac{1}{T}$.
Since $\cos \theta,\sin \theta >0$ and $\theta_M\in (0,\pi)$, we have
\[\textstyle
1+\cos 2\delta_M =1+\cos (\theta_M -2\theta) \geqslant 1+\cos2\theta\cos\theta_M \geqslant 1 -\vert \cos
\theta_M \vert = 1-\sqrt{1-\sin^2 \theta_M} \gg \frac{1}{T^2} .
\]
As $U_M-1\ll \frac{1}{T^2}$, it follows that
$ \frac{U_M-1}{1+\cos 2\delta_M}\ll 1 $, contradiction.

We have thus shown that $|\delta_M| \leqslant |\tan \delta_M|\ll\frac{1}{T}$, or more precisely there exists $\Theta_0(\xi)$
continuous in $\xi$ such that $| \delta_M | \leqslant \frac{\Theta_0(\xi)}{T}$.

\noindent{Case I.} $Y>X$. Then $0<\frac{\theta_M}{2} < \frac{\pi}{4}$ and $Z =\sqrt{XY-1} <Y$.
Since $| \delta_M | \ll \frac{1}{T}\ll Q^{-2c_0}$, one has $0<\theta <\frac{\pi}{3}$ for large $Q$.
Employing the formula $\tan( \frac{\theta_M}{2}) =\frac{Z}{Y-\epsilon_T}$
with $\epsilon_T$ as in \eqref{3.1}, we infer
\begin{equation}\label{7.8}
\bigg| \frac{AC+BD}{C^2+D^2-\epsilon_T} -\frac{q^\prime}{q}\bigg|
=\vert \tan \delta_M\vert \bigg| 1+\tan\theta \tan \bigg( \frac{\theta_M}{2}\bigg) \bigg|
\ll \frac{1}{T} .
\end{equation}
Combining \eqref{7.8} with $0< \frac{Z}{Y-\epsilon_T}-\frac{Z}{Y} \ll \frac{1}{T}$
and with $\big| \frac{Z}{Y}-\frac{A+B}{C+D}\big| \leqslant \frac{1}{C^2+D^2}\ll \frac{1}{T}$, we arrive at
\begin{equation}\label{7.9}
\bigg| \frac{A+B}{C+D}-\frac{q^\prime}{q} \bigg| \ll \frac{1}{T} \leqslant Q^{-2c_0} .
\end{equation}
If nonzero, the left-hand side in \eqref{7.9} must be $\geqslant \frac{1}{q(C+D)}$. But
$q(C+D) \leqslant q\sqrt{2(C^2+D^2)} \leqslant Q\sqrt{2}$, and so $Q^{2c_0} \ll Q$, contradiction.
It remains that $q=C+D$ and $q^\prime =A+B$, which again is not possible because
$Q^{2c_0} \leqslant (C+D)^2 =q(C+D)\leqslant Q\sqrt{2}$.

\noindent{Case II.} $X>Y$. Then $\frac{\pi}{4} <\frac{\theta_M}{2} < \frac{\pi}{2}$ and $Y\leqslant \sqrt{XY-1}=Z$.
As $\vert \delta_M \vert \ll Q^{-2c_0}$, we must have $0<\frac{\pi}{2} -\theta <\frac{\pi}{3}$ for large values of $Q$.
This time we have
\begin{equation*}
\begin{split}
\bigg| \frac{Y-\varepsilon_T}{Z}-\frac{q}{q^\prime} \bigg| & =
\bigg| \tan \bigg( \frac{\pi}{2} -\frac{\theta_M}{2}\bigg) -\tan \bigg( \frac{\pi}{2} -\theta \bigg) \bigg|  \\
& =\vert \tan \delta_M\vert \bigg| 1+\tan \bigg( \frac{\pi}{2}-\frac{\theta_M}{2}\bigg)
\tan \bigg( \frac{\pi}{2}-\theta \bigg) \bigg| \leqslant (1+\sqrt{3}) \vert \tan\delta_M \vert \ll \frac{1}{T},
\end{split}
\end{equation*}
which leads (use $D\geqslant C \Longleftrightarrow B\geqslant A)$ to
\begin{equation}\label{7.10}
\begin{split}
\bigg| \frac{C+D}{A+B}-\frac{q}{q^\prime} \bigg| & \ll \frac{1}{T} +\frac{\epsilon_T}{Z} + \bigg| \frac{Y}{Z}-\frac{C+D}{A+B}\bigg|
\ll \frac{1}{T} +\frac{\vert D-C\vert}{(A+B)(AC+BD)} \\
& \leqslant \frac{1}{T} +\frac{1}{(A+B)^2} \leqslant \frac{1}{T} +\frac{1}{X} \ll \frac{1}{T} \ll Q^{-2c_0}  .
\end{split}
\end{equation}
As in Case I this is not possible because $q^\prime (A+B)\leqslant q^\prime \sqrt{2X} \leqslant Q\sqrt{2}$
and $(A+B)^2 \geqslant Q^{2c_0}$.
\end{proof}

Our next goal is to apply Lemma \ref{L7}, assuming $Y_M < X_M \ll Q^{2c_0}$ and taking $r=1$, to the set
$\Omega=\Omega_{M,q^\prime,\xi}$ of pairs $(u,v)\in (0,Q]\times (0,q^\prime]$ that satisfy
\begin{equation}\label{7.11}
\vert \Xi_M (q^\prime,u)\vert \leqslant \frac{\xi}{Q^2}  \quad \mbox{\rm and} \quad
q^{\prime 2} X_M +  u^2 Y_M + 2 u q^\prime Z_M\leqslant \frac{Q^2 q^{\prime 2}}{v^2+q^{\prime 2}} .
\end{equation}

The next related statement will be useful:
\begin{lemma}\label{L10}
There exist continuous functions $T_0(\xi)$ and $C(\xi)$ such that,
for any matrix $M\in {\mathfrak S}$ with $Y_M <X_M$ and $T=\| M\|^2 >T_0(\xi)$,
the projection on the first coordinate of the set $\Omega_{M,q^\prime,\xi}$
is contained in the interval $(0,C(\xi)q^\prime]$.
%In particular, if $(q^\prime,q)$ is as in \eqref{7.2}, then $q\leqslant C(\xi)q^\prime$.
\end{lemma}

\begin{proof}
Using polar coordinates $(u,q^\prime)=(r\cos\theta,r\sin\theta)$, $\theta \in (0,\frac{\pi}{2})$, we see that inequalities
\eqref{7.11} imply \eqref{7.7}. This shows that for the purpose of this lemma we can replace
$\Omega_{M,q^\prime,\xi}$ by the set of $(u,v)\in (0,Q]\times (0,q^\prime]$ satisfying \eqref{7.7}.
Therefore we can use all estimates from the first part of the proof of Lemma \ref{L9} (because they only rely on \eqref{7.7},
the integrality of $q$ being used only at the end).

Note also that $Y=Y_M<X=X_M$ and $Z^2=XY -1$ yield $Y \leqslant Z$.
Replacing $q$ by $u$ in the first part of the proof of Lemma \ref{L9}, so that $\tan \theta=\frac{u}{q^\prime}$, $\theta\in (0,\frac{\pi}{2})$,
we see (cf. last line before Case 1) that $\vert \delta_M\vert \leqslant \frac{\Theta (\xi)}{T}$ for some continuous function $\Theta$.
Next we look into the first estimates in Case 2 and see that there exists $T_0(\xi)$ depending continuously on $\xi$ such that,
for any $M$ with $T=\| M\|^2 >T_0(\xi)$, one has $0<\frac{\pi}{2}-\theta <\frac{\pi}{3}$ and $| \frac{u}{q^\prime} -\frac{Y-\epsilon_T}{Z} | \leqslant
(1+\sqrt{3})| \tan \delta_M|$. In conjunction with the $\delta_M$-bound, this shows the existence of a continuous function $C_0(\xi)$ such that
$| u-\frac{Y-\epsilon_T}{Z} q^\prime| \leqslant C_0(\xi) q^\prime$, showing that $u\leqslant (1+C_0(\xi))q^\prime$.
\end{proof}

Although this will not be used in this paper, we remark that if $\gamma$ is as in
\eqref{7.2}, then \eqref{7.4}
%subsequently \eqref{7.6} and \eqref{7.7}
is satisfied by the triple $(q^\prime,q,p^\prime)$ with the pair $(Q,\xi)$ replaced by
 $(Q+\sqrt{2}Q^{c_0},\xi (1+\sqrt{2} Q^{c_0-1})^2 )$, by the proof of \eqref{7.5}.
Therefore Lemma \ref{L10} shows that $\frac{q}{q^\prime} \ll_\xi 1$ (with a different
implicit constant than $C(\xi)$ from Lemma \ref{L10}).

Next notice that, as $Q\rightarrow\infty$,
\begin{equation}\label{7.12}
\sum\limits_{\substack{M\in {\mathfrak S} \\ \max\{ X_M,Y_M\} \leqslant Q^{2c_0}}} \hspace{-15pt} \max\{ X_M,Y_M\}^{-\sigma} \ll_\sigma Q^{(2-2\sigma)c_0} ,\quad 0<\sigma <1.
\end{equation}
This follows immediately from\footnote{Here $A$ and $B$ determine uniquely the matrix $M=\left( \begin{smallmatrix}
A & B \\ C & D \end{smallmatrix}\right)$.}
\begin{equation*}
\sum\limits_{\substack{M\in {\mathfrak S} \\ Y_M <X_M \leqslant Q^{2c_0}}} X_M^{-\sigma} \leqslant \sum_{1\leqslant A^2+B^2 \leqslant Q^{2c_0}} (A^2+B^2)^{-\sigma} \leqslant
\iint_{x^2+y^2 \leqslant 2Q^{2c_0}} (x^2+y^2)^{-\sigma}dx\, dy \ll_\sigma Q^{(2-2\sigma)c_0}  .
\end{equation*}

Assume now that $Y_M < X_M \leqslant Q^{2c_0}$. When $T=\| M\|^2 >T_0(\xi)$ we apply Lemma \ref{L10}.
The definition of $\Omega$, seen after some obvious scaling as a section subset
in the body $S_{M,\xi}$ defined in \eqref{7.15} below,
shows that the range of $u$ consists of a union of intervals in $[0,Q]$ with a (universally) bounded number of components and of total
Lebesgue measure $\ll_\xi q^\prime$. This gives $A(\Omega)\ll_\xi \frac{Qq^\prime}{\sqrt{X_M}}$ and
$\ell (\partial \Omega) \ll_\xi q^\prime +q^\prime \ll \frac{Q}{\sqrt{X_M}}$.
Taking $L=q^{\prime 5/6}$ we find $Q\gg X_M^{1/2} q^{\prime 1/6}$, and the error provided by Lemma \ref{L7} is
$\EE_{\Omega,L,q^\prime} \ll_\varepsilon Q q^{\prime -1/6+\varepsilon} X_M^{-1/2}$.
Note also that in this case $A\geqslant C$ and $B\geqslant D$.
As a result, applying \eqref{7.12} with $\sigma=\frac{11}{12}$, the error is seen to add up to
\begin{equation*}
\sum\limits_{\substack{A^2+B^2 \leqslant Q^{2c_0} \\ \| M\|^2 >T_0(\xi)}}\sum_{q^\prime\leqslant Q/\sqrt{X_M}} \EE_{\Omega,q^{\prime 5/6},q^\prime}\ll_\varepsilon
Q \sum_{A^2+B^2 \leqslant Q^{2c_0}} \frac{1}{X_M^{1/2}} \bigg( \frac{Q}{X_M^{1/2}}\bigg)^{5/6+\varepsilon}
\ll_\varepsilon Q^{(11+c_0)/6+\varepsilon}.
\end{equation*}
Lemma \ref{L7} now provides
\begin{equation}\label{7.13}
\widetilde{\NN}_M(Q,\xi) =\sum_{1\leqslant q^\prime \leqslant Q/\sqrt{X_M}} \frac{\varphi(q^\prime)}{q^{\prime 2}} A(\Omega_{M,q^\prime,\xi})
+O_\veps ( Q^{(11+c_0)/6+\varepsilon}).
\end{equation}
The situation $\| M\|^2 \leqslant T_0(\xi)$ (in this case there are $O_\xi(1)$ choices for $M$) is directly
handled by Lemma \ref{L7}. The same choice for $L$ provides $\EE_{\Omega,q^{\prime 5/6},q^\prime} \ll_\varepsilon Qq^{\prime -1/6+\varepsilon}$.
These error terms sum up to $O_{\veps,\xi} (Q^{11/6+\veps})$ in this situation.

Next we will apply M\" obius summation (cf., e.g., \cite[Lemma 2.3]{BCZ0})  to the function
$h_1(q^\prime)=\frac{1}{q^\prime} A(\Omega_{M,q^\prime,\xi})$. Note that
$\frac{1}{Q}h_1(q^\prime)$ represents the area of the cross-section the body
\begin{equation}\label{7.15}
S_{M,\xi} := \left\{ (x,y,z)\in [0,1]^3 :  \vert \Xi_M (x,y)\vert \leqslant
\xi,\
x^2 X_M+y^2 Y_M+2xy Z_M \leqslant \frac{1}{1+z^2}\right\}
\end{equation}
by the plane $x=\frac{q^\prime}{Q}$. The intersection of the projection of $S_{M,\xi}$ onto
the plane $z=0$ with a vertical line $x=c$ is
bounded by a quartic and an ellipse, showing that the cross-section
function $c\mapsto  A_{M,\xi}(c):=\mathrm{Area}(S_{M,\xi}\cap  \{x=c\})$ is continuous and piecewise
$C^1$ on $[0,1]$
and the number of critical points of $A_{M,\xi}$ is bounded by a universal constant $C$ independently of $M$ and $\xi$.
The graph on the right of Figure \ref{Figure6} illustrates one of the possible cases that
can arrise, when $A_{M,\xi}(c)$ has the most number of critical points, showing that we
can take $C=3$.

\begin{figure}[ht]
\begin{center}
\includegraphics*[scale=0.6, bb=0 1 280 240]{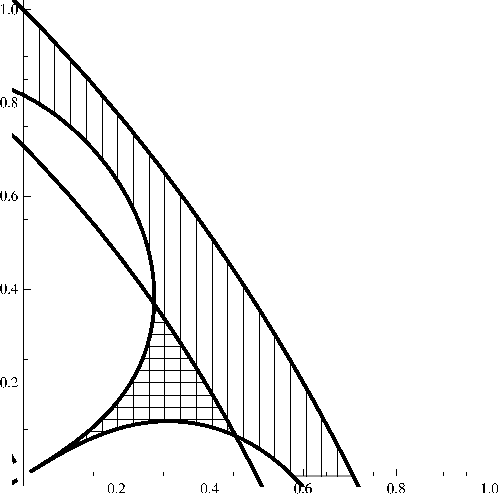}
\includegraphics*[scale=0.8, bb=0 1 240 200]{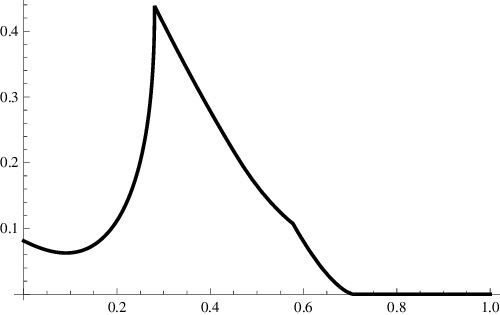}
\caption{The cross-sections of $S_{M,\xi}$ for $z=0$ (hashed vertically) and
$z=1$ (hashed horizontally), and the function $c\mapsto A_{M,\xi}(c)$,
for $M= R$ and $\xi=1.5$. }
\label{Figure6}
\end{center}
\end{figure}

In particular the total variation of $h_1$ on $[0,Q]$ is
$\leqslant (C+1) (\sup_{[0,Q]} h_1 -\inf_{[0,Q]} h_1) \ll
\| h_1\|_\infty \ll_\xi \frac{Q}{\sqrt{X_M}}$, and so we infer
\[
\sum_{1\leqslant q^\prime \leqslant Q/\sqrt{X_M}} \frac{\varphi(q^\prime)}{q^{\prime 2}}
A(\Omega_{M,q^\prime,\xi}) =\frac{1}{\zeta(2)}
\int_0^{Q/\sqrt{X_M}} h_1(q^\prime)\, dq^\prime +O \bigg( \frac{Q}{\sqrt{X_M}} \ln Q\bigg)
.
\]
Using also the change of variables $(q^\prime,u,v)=(Qx,Qy,Qxz)$,
$(x,y,z)\in [0,1]^3$, \eqref{7.13}, \eqref{7.5} and \eqref{7.12}, we find that the
contribution to $R_Q^\Cap (\xi)$ of
matrices $M$ with $Y_M<X_M$ is
\begin{equation}\label{7.14}
\begin{split}
\frac{1}{\zeta (2)}
\sum\limits_{\substack{M\in {\mathfrak S} \\ Y_M <X_M\leqslant Q^{2c_0} }} &
\Bigg( \int_0^{Q/\sqrt{X_M}}\hspace{-5pt}  A(
\Omega_{M,q^\prime,\xi})\frac{dq^\prime}{q^\prime} +O\Big(
\frac{Q\ln Q}{X_M^{1/2}} \Big)\Bigg)
 +O_{\veps,\xi} \big(Q^{(11+c_0)/6+\veps}\big)   \\
& = \frac{Q^2}{\zeta(2)}
\sum\limits_{\substack{M\in {\mathfrak S} \\ Y_M < X_M \leqslant Q^{2c_0} }}
\operatorname{Vol} (S_{M,\xi})+O_{\veps,\xi} \big( Q^{1+c_0+\veps}
+Q^{(11+c_0)/6+\veps}\big) .
\end{split}
\end{equation}

With $\eta =\left( \begin{smallmatrix} 0 & 1 \\ 1 & 0 \end{smallmatrix} \right)$ notice
the following important symmetries:
\begin{equation}\label{7.16}
\eta M\eta=\left( \begin{smallmatrix}
D & C \\ B & A \end{smallmatrix}\right) \quad \mbox{\rm and} \quad
\Xi_{\eta M\eta} (y,x) = -\Xi_M (x,y),
\end{equation}
showing that the reflection $(x,y,z)\mapsto (y,x,z)$ maps $S_{M,\xi}$
bijectively onto $S_{\eta M \eta,\xi}$.

The situation $X_M<Y_M$ is handled similarly using \eqref{7.16},
which results in reversing the roles of $q$ and $q^\prime$ with Lemma \ref{L7} applied for $r=-1$.

Next we give upper bounds for $\operatorname{Vol}(S_{M,\xi})$. Take $(x,y,z)=(r\cos t,r\sin t,z) \in S_{M,\xi}$.
The proof of \eqref{7.9} and \eqref{7.10} does not use the integrality of $q^\prime$ and $q$, so denoting
$\omega_M =\frac{C+D}{A+B}<1$ if $Y_M <X_M$ and $\omega_M=\frac{A+B}{C+D}<1$ if $X_M <Y_M$, we find
$y\ll x\ll X_M^{-1/2}\ll T^{-1}$ and
$\big| \frac{y}{x}-\omega_M\big| \ll \frac{1}{T}$ in the former case, and respectively
$x\ll y \ll Y_M^{-1/2}\ll T^{-1}$ and $\big| \frac{x}{y}-\omega_M\big| \ll \frac{1}{T}$ in the latter case.
Writing the area in polar coordinates
we find $r^2 \ll T^{-1}$ and
\begin{equation}\label{7.17}
\begin{split}
\operatorname{Vol} (S_{M,\xi}) & \leqslant A \big( \big\{ (x,y)\in [0,1]^2:
\exists\, z\in [0,1],\ (x,y,z)\in S_{M,\xi}\big\}\big) \\ & \leqslant \frac{1}{2}
\int_{\omega_M -\xi T_M^{-1}}^{\omega_M+\xi T_M^{-1}} 2T_M^{-1}\ dt
=\frac{2\xi}{T_M^2} =\frac{2\xi}{\| M\|^4} .
\end{split}
\end{equation}
The bound \eqref{7.17} and a reasoning similar to the proof of \eqref{7.12} yields
\begin{equation}\label{7.18}
\sum\limits_{M\in {\mathfrak S}}
\operatorname{Vol} (S_{M,\xi}) < \infty \quad \mbox{\rm and} \quad
\sum\limits_{\substack{M\in {\mathfrak S} \\ \max \{ X_M,Y_M\} \geqslant Q^{2c_0}}}
\operatorname{Vol} (S_{M,\xi}) \ll_\xi Q^{-2c_0} .
\end{equation}
From \eqref{7.14}, \eqref{7.18} and $c_0 \in ( \frac{1}{2},1)$, we infer
\begin{equation}\label{7.19}
R_Q^\Cap (\xi)= \frac{Q^2}{\zeta(2)}
\sum\limits_{\substack{M\in {\mathfrak S}}}
\operatorname{Vol} (S_{M,\xi}) +O_\veps (Q^{(11+c_0)/6+\veps}) .
\end{equation}

The volume of $S_{M,\xi}$ can be evaluated in closed form using the substitution
$z=\tan t$:
\begin{equation}\label{7.20}
\operatorname{Vol} (S_{M,\xi} ) = \int_{0}^{\pi/4} B_{M} (\xi,t)
\frac{dt}{\cos^2 t} ,
\end{equation}
where $B_{M}(\xi,t)$ is the area of the region
\begin{equation}\label{7.21}
 \bigg\{ (r \cos \theta,
r \sin \theta)\in [0,1]^2: \frac{1}{\xi}\cdot \frac{|\sin
(2\theta-\theta_M)|}{U_T+\cos (2\theta -\theta_M)}\leqslant r^2 \leqslant
\frac{1}{\sqrt{T^2-4}}\cdot
\frac{2\cos^2 t}{U_T+\cos(2\theta -\theta_M)} \bigg\},
\end{equation}
with $\theta_M\in (0,\frac{\pi}{2} )$ having $\sin
\theta_M=\frac{2Z_M}{\sqrt{T^2-4}}$ and
$U_T=\frac{T}{\sqrt{T^2-4}}$ (for brevity we write $T=T_M$).

The following elementary fact will be useful to prove the differentiability of the
volumes as functions of $\xi$.

\begin{lemma}\label{L11}
Assuming $G,H:K\rightarrow \R$ continuous functions on a compact set $K\subset \R^k$
and denoting $x_+=\max\{ x,0\}$, the formula
\begin{equation*}
V(\xi) := \int_K ( \xi -G(v))_+ H(v)\, dv,\quad \xi\in\R ,
\end{equation*}
defines a $C^1$ map on $\R$ and
\begin{equation*}
V^\prime (\xi) = \int_{G<\xi} H(v)\, dv .
\end{equation*}
\end{lemma}
Using equation \eqref{7.20} we find
\begin{equation}\label{7.22}
\operatorname{Vol} (S_{M,\xi} ) =\frac{1}{2} \int_0^{\pi/4} dt \int_{0}^{\pi/2} d\theta\,
 \frac{\big(2/\sqrt{T^2-4}- |\sin
(2\theta-\theta_M)|/(\xi\cos^2 t)\big)_+}{U_T+\cos (2\theta-\theta_M)}
\end{equation}
and applying Lemma \ref{L11} we obtain

\begin{cor}\label{C12} The function
$\xi\mapsto\operatorname{Vol} (S_{M,\xi} )$ is $C^1$.
%and $\vert g_M^\prime (\xi)\vert \ll T^{-2}$ when $\xi \leqslant Z_M$.
\end{cor}
For a smaller range for $\xi$ we have the following explicit formula.

\begin{lemma}\label{L13}
Suppose that $\xi\leqslant Z_M$. The volume of
$S_{M,\xi}$ only depends on $\xi$ and $T=\|M\|^2$:
\[
\operatorname{Vol} (S_{M,\xi} )= \int_0^{\pi/4} \hspace{-3pt}  \tan^{-1}\bigg(
 \frac{\sqrt{\Delta}-\sqrt{\Delta-4\xi^2\cos^4 t}}{2\alpha\xi \cos^2 t}\bigg)
+\frac{1}{2\xi\cos^2 t}\ln\bigg(1-\frac{\sqrt{\Delta}-\sqrt{\Delta-4\xi^2\cos^4
t}}{2\alpha}\bigg) dt,
\]
where $\Delta=T^2-4$ and $\alpha=\frac{1}{2} (T+\sqrt{T^2-4})$.
\end{lemma}

\begin{proof}The two polar curves in
\eqref{7.21} intersect for
$|\sin(2\theta-\theta_M)|
=\frac{2\xi}{\sqrt{T^2-4}}\cos^2 t$, that is for $\theta_\pm=\frac{\theta_M}{2}\pm
\alpha$ with $\alpha=\alpha(\xi,t)\in ( 0,\frac{\pi}{4})$ such that $\sin
2\alpha= \frac{2\xi}{\sqrt{T^2-4}}\cos^2 t$. Since $\sin \theta_M=\frac{2Z}{\sqrt{T^2-4}}$,
the assumption $ \xi\leqslant Z$ ensures $\alpha<\theta_M$. Thus
$\theta_{\pm}\in [ 0,\frac{\pi}{2} )$, and a change of variables
$\theta=\frac{\theta_M}{2}+u$ yields
\[
B_{M,\xi}(t)=\frac{1}{2}\int_{-\alpha}^{\alpha}
\bigg( \frac{2\cos^2 t}{\sqrt{T^2-4}}\cdot
\frac{1}{U_T+\cos(2u)}-\frac{|\sin (2u)|}{\xi(U_T+\cos (2u))}\bigg) d u.
\]
The integrand is even and both integrals can be computed exactly, yielding the formula
above.
\end{proof}
In particular Lemma \ref{L13} yields $\operatorname{Vol} (S_{M,\xi}) \ll \frac{\xi}{T^2}$,
providing an alternative proof for \eqref{7.17}.

\subsection{Exterior arcs}
Referring to the notation of Section \ref{sec5.2}, we first replace
the inequalities $p^2+p^{\prime 2}+q^2+q^{\prime 2} \leqslant Q^2$ and
$p_\ell^2+q_\ell^2 +(K p_\ell-p_{\ell-1})^2 +(Kq_\ell -q_{\ell-1})^2 \leqslant Q^2$
in \eqref{5.7} by simpler ones. Using $p^\prime q-pq^\prime=1$ we can replace $p$ by
$\frac{p^\prime q}{q^\prime}$ in the former, while $p_{\ell-1}$ can be replaced by
$\frac{p_\ell q_{\ell-1}}{q_\ell}$ in the latter. As a result these two inequalities can be substituted in
\eqref{5.7} by
\begin{equation}\label{7.23}
\begin{cases}
\vspace{0.15cm}
\displaystyle \bigg( 1+\frac{p^{\prime 2}}{q^{\prime 2}} \bigg) (q^2+q^{\prime 2}) \leqslant Q^2 \big( 1+O
(Q^{-1}) \big) \\
\displaystyle \bigg( 1+\frac{p_\ell^2}{q_\ell^2}\bigg) \big( q_\ell^2+(Kq_\ell -q_{\ell-1})^2\big) \leqslant Q^2
\big( 1+O (Q^{-1}) \big) .
\end{cases}
\end{equation}
Since $\frac{p_\ell}{q_\ell} =\frac{p^\prime}{q^\prime}+O( \frac{\ell}{Q})$ and $q_\ell^2+(Kq_\ell
-q_{\ell-1})^2\leqslant 2Q^2$,
the second inequality in \eqref{7.23} can be also written as
\begin{equation*}
\bigg( 1+\frac{p^{\prime 2}}{q^{\prime 2}} \bigg) \big( q_\ell^2+(Kq_\ell -q_{\ell-1})^2\big) \leqslant Q^2
\big( 1+O (Q^{-1}) \big) ,
\end{equation*}
leading to
\begin{equation*}
R_Q^{\cap \,\cap} (\xi) =\sum\limits_{\substack{\ell\in [0,\xi) \\ K\in [1,\xi)}}\sum_{q'<Q}  \NN^{\cap\,
\cap}_{Q+O(Q^{1/2}),q^\prime,K,\ell} (\xi),
\end{equation*}
where $\NN_{Q,q^\prime,K,\ell}^{\cap\,\cap} (\xi)$ denotes the number of integer lattice points $(p^\prime,q)$ such that
\begin{equation}\label{7.24}
\begin{cases} \vspace{0.1cm}
& 0\leqslant p^\prime \leqslant q^\prime,\quad 0\leqslant q\leqslant Q, \quad
p^\prime q \equiv 1 \hspace{-6pt} \pmod{q^\prime},\quad
 0< Kq_\ell -q_{\ell -1} \leqslant Q \\  \vspace{0.1cm}
& \displaystyle \Upsilon_{\ell,K} \bigg( \frac{q}{Q},\frac{q^\prime}{Q}\bigg)
\leqslant \xi ,   \quad
p^{\prime 2}+q^{\prime 2} \leqslant \frac{Q^2 q^{\prime 2}}{\max\{ q^2+q^{\prime 2},
q_\ell^2 +(Kq_\ell -q_{\ell -1})^2\}}.
\end{cases}
\end{equation}

Applying Lemma \ref{L7} to the set $\Omega =\Omega^{\cap\, \cap}_{q^\prime,K,\ell,\xi}$ of elements
$(u,v)$ for which
\begin{equation*}
\begin{cases} \vspace{0.1cm}
& \displaystyle u\in [0,Q],\ \ v\in [0,q^\prime],\ \
 L_i \bigg( \frac{u}{Q},\frac{q^\prime}{Q}\bigg) >0,\ i=0,1,\ldots,\ell \\
& \displaystyle 0< KL_\ell \bigg(\frac{u}{Q},\frac{q^\prime}{Q}\bigg) -L_{\ell-1}
\bigg( \frac{u}{Q},\frac{q^\prime}{Q}\bigg) \leqslant 1,
\quad \Upsilon_{\ell,K} \bigg( \frac{u}{Q},\frac{q^\prime}{Q}\bigg) \leqslant \xi  \\
& \displaystyle v^2+q^{\prime 2} \leqslant \frac{Q^2 q^{\prime 2}}{\max \big\{ u^2+q^{\prime 2},
Q^2 L_\ell^2 \big(\frac{u}{Q},\frac{q^\prime}{Q}\big) +Q^2 \big( KL_\ell \big( \frac{u}{Q},\frac{q^\prime}{Q}\big)
- L_{\ell-1} \big( \frac{u}{Q},\frac{q^\prime}{Q}\big)\big)^2\big\}} ,
\end{cases}
\end{equation*}
with $A(\Omega) \leqslant Qq^\prime$, $\ell (\partial \Omega) \ll Q$, $L=q^{\prime 5/6}$, we find
\begin{equation*}
\NN^{\cap\,\cap}_{Q,q^\prime,K,\ell} (\xi) = \frac{\varphi (q^\prime)}{q^\prime} \cdot
\frac{A\big( \Omega^{\cap\,\cap}_{q^\prime ,K,\ell,\xi} \big)}{q^\prime} +O_\veps ( Qq^{\prime -1/6+\veps}) .
\end{equation*}
This leads in turn to
\begin{equation*}
R_Q^{\cap\,\cap}(\xi) = \MM^{\cap\,\cap}_Q (\xi) +O_{\xi,\veps} ( Q^{11/6+\veps}),
\end{equation*}
where
\begin{equation*}
\MM^{\cap\,\cap}_Q (\xi) = \sum\limits_{\substack{\ell\in [0,\xi) \\ K \in [1,\xi)}} \sum_{q^\prime \leqslant Q}
\frac{\varphi(q^\prime)}{q^\prime}\cdot \frac{A(\Omega_{q^\prime,K,\ell,\xi}^{\cap\,\cap})}{q^\prime} .
\end{equation*}

For fixed integers $K\in [1,\xi)$, $\ell \in [0,\xi)$, consider
the subset $T_{K,\ell,\xi}$ of $[0,1]^3$ defined as
\begin{equation}\label{7.25}
\left\{ (x,y,z) \in [0,1]^3: \begin{matrix}
0<L_{\ell+1}(x,y)=K L_\ell (x,y)-L_{\ell-1}(x,y)\leqslant 1, \Upsilon_{\ell,K} (x,y) \leqslant \xi
\\ \displaystyle \max \left\{ x^2+y^2 , L_\ell^2 (x,y)+L_{\ell +1}^2 (x,y) \right\} \leqslant \frac{1}{1+z^2}
\end{matrix} \right\} ,
\end{equation}
with $L_i$ and $\Upsilon_{\ell,K}$ as in \eqref{5.5} and \eqref{5.6}.

M\" obius summation is now applied to $h_2(q^\prime)=\frac{1}{q^\prime} A\big(
\Omega_{q^\prime,K,\ell,\xi}^{\cap\,\cap}\big)$. The quantity $\frac{1}{Q} h_2 (q^\prime)$
represents the area of the cross-section of the body $T_{K,\ell,\xi}$ by the plane $x=\frac{q^\prime}{Q}$.
This shows that $h_2$ is continuous and piecewise $C^1$ on $[0,Q]$ and furthermore the number of critical points
of $h_2$ is bounded uniformly in $\xi$ (and independently of $Q$). Hence the total
variation of $h_2$ on $[0,Q]$ is $\ll_\xi \| h_2\|_\infty \leqslant Q$.
Employing also the change of variables $(q^\prime,u,v)=(Qx,Qy,Qxz)$, $(x,y,z)\in [0,1]^3$ we find
\begin{equation*}
\MM^{\cap\,\cap}_Q (\xi) = \frac{1}{\zeta (2)} \sum\limits_{\substack{\ell\in [0,\xi) \\  K \in [1,\xi)}}
\Bigg( \int_{0}^Q \frac{dq^\prime}{q^\prime} A\big( \Omega^{\cap\,\cap}_{q^\prime,K,\ell,\xi}\big) +O(Q) \Bigg)
 = \frac{Q^2}{\zeta (2)} \sum\limits_{\substack{\ell\in [0,\xi) \\  K \in [1,\xi)}}
\operatorname{Vol} (T_{K,\ell,\xi}) +O_\xi (Q),
\end{equation*}
and so
\begin{equation}\label{7.26}
R_Q^{\cap\,\cap} (\xi) =\frac{Q^2}{\zeta(2)} \sum\limits_{\substack{\ell\in [0,\xi) \\  K \in [1,\xi)}}
\operatorname{Vol} (T_{K,\ell,\xi})+O_{\xi,\veps} (Q^{11/6+\veps}) .
\end{equation}

To show that $\xi\mapsto\operatorname{Vol}(T_{K,\ell,\xi})$ is $C^1$ on
$[1,\infty)$, we change variables $(x,y,z)=(\cos\theta,\sin\theta, \tan t)$ to obtain
\begin{equation}\label{7.27}
\operatorname{Vol} (T_{K,\ell,\xi}) =\int_0^{\pi/4} A_{K,\ell} (\xi,t)\,  \frac{dt}{\cos^2 t},
\end{equation}
where $A_{K,\ell} (\xi,t)$ is the area of the region defined by \eqref{1.3}.
Now notice that $K_i(x,y)\leqslant \xi$ when $1\leqslant i\leqslant \ell$,
as a result of (omitting the arguments of the functions)
\[ K_i=\frac{L_i +L_{i-2}}{L_{i-1}}\leqslant
\frac{1}{L_{i-2} L_{i-1}}+\frac{1}{L_{i-1} L_{i}} < \Upsilon_{\ell,K} \leqslant \xi.
\]
Similarly, $K_1=\frac{L_{-1}+L_1}{L_0}\leqslant \frac{L_{-1}}{L_0}+\frac{1}{L_0 L_1}<  \Upsilon_{\ell,K} \leqslant
\xi$.
Thus the projection of $T_{K,\ell,\xi}$ on the first
two coordinates is included into the union of disjoint cylinders
$\TT_{\mathbf k} :=\TT_{k_1} \cap T^{-1} \TT_{k_2} \cap \ldots \cap T^{-\ell+1} \TT_{k_\ell}$ with
$\TT_k=\{ (x,y): K_1 (x,y)=k\}$ and
${\mathbf k}=(k_1,\ldots ,k_\ell)\in [1,\xi)^{\ell}$. On each set
$\TT_{\mathbf k}$ all maps
$L_1 , \ldots  , L_\ell , L_{\ell+1}$ are linear, say $L_i(x,y) = A_i x + B_i y$ with
integers $A_i,B_i$ depending only on $k_1,\ldots,k_i$ for $i\leqslant \ell$ and
$A_{\ell+1},B_{\ell+1}$ depending only on ${\mathbf k}$ and $K$. Therefore the function
$F_{K,\ell} (\theta)$ is continuous on each region $\TT_{\mathbf k}$, and applying
Lemma \ref{L11} we conclude that the function $\xi\mapsto\operatorname{Vol}(T_{K,\ell,\xi})$ is $C^1$ on $[1,
\infty]$, being a sum of $[\xi]^\ell$ volumes, each of which $C^1$ as functions of $\xi$.

\begin{remark}\label{R14}
{\em The region $T_{K,\ell,\xi}$ can be simplified further. For
each integer $J \in [1,\xi)$, the map
$$\Psi_{J}:(u,v)\mapsto (J L_\ell(u,v)-L_{\ell-1}(u,v),L_\ell(u,v) )$$
is an area preserving injection on $\TT$, since it is the composition of $T^{\ell}$ in
\eqref{5.5} followed by
the linear transformation $(u,v)\mapsto (J v-u, v)$. Note that under
this map (omitting the arguments $(u,v)$ of the functions below):
\[
L_1\rightarrow \left[\frac{1+J L_\ell-L_{\ell-1}}{L_\ell}\right] -(J
L_\ell-L_{\ell-1})=L_{\ell-1}
\]
(using $L_{\ell-1}+L_{\ell}>1$), and by induction it follows similarly that
$L_i\rightarrow L_{\ell-i}$ for $0\leqslant i \leqslant \ell$. Also we have
that $\Psi_{J}(u,v)=(x,y)\in [0,1]^2$ if and only if $x=J L_\ell-L_{\ell-1}\in [0,1]$ and
$J= [\frac{1+x}{y} ]$.

Let us decompose the region $T_{K,\ell,\xi}$ into a disjoint union of regions
$T_{K,J;\ell,\xi}$, $1\leqslant J <\xi$, obtained by adding the condition
$[\frac{1+x}{y} ]=J$. By the discussion of the previous paragraph, the
map $(\Psi_{J},\mathrm{Id}_z)$ is a volume preserving bijection taking $U_{K,J;\ell,\xi}$
onto $T_{K,J;\ell,\xi}$, where
\[
U_{K,J;\ell,\xi} :=\left\{ (x,y,z) \in [0,1]^3: \begin{matrix} x+y>1, \quad
J L_\ell -L_{\ell-1}>0, \quad KL_0-L_1>0,\quad \Upsilon_{\ell,K,J} \leqslant
\xi  \\ \displaystyle L_0^2+(KL_0-L_1)^2 \leqslant
\frac{1}{1+z^2},\quad  L_\ell^2 +(J L_\ell-L_{\ell-1})^2  \leqslant
\frac{1}{1+z^2}
\end{matrix} \right\}.
\]
Here $L_i=L_i(x,y)$ and $\Upsilon_{\ell,K,J} (x,y)=\frac{J
L_\ell-L_{\ell-1}}{L_\ell (L_{\ell}^2+(J L_\ell-L_{\ell-1})^2)}
+\sum_{i=1}^\ell \frac{1}{L_{i-1}L_i}
+\frac{K L_0-L_1}{L_0 ( L_0^2+(K L_0-L_1)^2)}$.

For $\alpha \geqslant 1$, the transformation $(\Psi_\alpha,\mathrm{Id}_z) $ maps
bijectively the part of $U_{K,J;\ell,\xi}$  for which $[\frac{1+L_{\ell-1}}{L_\ell}
]=\alpha $ onto the part of
$U_{J,K;\ell,\xi}$ for which $[\frac{1+x}{y} ]=\alpha$. Therefore
$\operatorname{Vol}
(U_{K,J;\ell,\xi})=\operatorname{Vol} (U_{J,K;\ell,\xi})$ and the sum of volumes
appearing in \eqref{7.28} can be written more symmetrically:
\[
\sum_{K\in [1,\xi)} \operatorname{Vol} (T_{K,\ell,\xi})=\sum_{K,J\in [1,\xi)} \operatorname{Vol}
(U_{K,J;\ell,\xi}).
\]
As an example of using this formula, if $1<\xi\leqslant 2$ and $\ell=1$, we can only have
$K=J=1$ and the inequalities $J L_1-L_0>0$, $KL_0-L_1>0$ cannot be both satisfied, so
$U_{1,1;1,\xi}$ is empty. Therefore the only contribution from the $T$ bodies in \eqref{7.28}
comes from $T_{1,0,\xi}$ if $\xi\in (1,2]$.
}
\end{remark}

We can now prove the main theorem regarding the pair correlation of the quantities $\tan(\frac{\theta_\gamma}{2})$.
\begin{theorem}\label{T2}
The pair correlation measure $R_2^{\mathfrak T}$ exists on $[0,\infty)$. It is given by
the $C^1$ function
\begin{equation}\label{7.28}
R_2^{\mathfrak T} \bigg( \frac{3}{8}\xi\bigg) = \frac{8}{3\zeta (2)} \Bigg(
\sum\limits_{M\in {\mathfrak
S}} \operatorname{Vol} (S_{M,\xi}) +\sum_{\ell \in [0,\xi)}
\sum_{K\in [1,\xi)} \operatorname{Vol} (T_{K,\ell,\xi}) \Bigg) ,
\end{equation}
where the three-dimensional bodies $S_{M,\xi}$ are defined in \eqref{7.15} and the bodies
$T_{K,\ell,\xi}$ are defined in \eqref{7.25}.
\end{theorem}

\begin{proof}
By \eqref{7.19},and \eqref{7.26}, with $c_0\in ( \frac{1}{2},1)$ and $G(\xi)$ denoting the
sum of all volumes in \eqref{7.28}, we infer
\begin{equation}\label{7.29}
\RR_Q^\Phi (\xi) =\frac{Q^2}{\zeta(2)} \,G ( \xi ) +O_{\xi,\veps} ( Q^{(11+c_0)/6+\veps}) .
\end{equation}
It follows that the function $G$ is $C^1$ on $[0,\infty)$ as a result of $\xi\mapsto
\operatorname{Vol}(S_{M,\xi})$  being $C^1$ on $[0,\infty)$, and  of
$\xi\mapsto \operatorname{Vol}(T_{K,\ell,\xi})$ being $C^1$ on $[1,\infty)$.
Corollary \ref{C4} and \eqref{7.29} now yield, for $\beta\in ( \frac{2}{3},1)$,
\begin{equation*}
\RR_Q^\Psi (\xi)=\frac{Q^2}{\zeta(2)}\,\Big( G\big( \xi+O(Q^{2-3\beta})\big) + G\big( O(Q^{2-3\beta})\big) \Big)
+O_{\xi,\veps} ( Q^{1+\beta}\ln Q +Q^{(11+c_0)/6+\veps} ).
\end{equation*}
Employing again the differentiability of $G$ and $G(0)=0$, and taking $\beta=\frac{3}{4}$,
$c_0=\frac{1}{2}+\veps$, this provides
\begin{equation}\label{7.30}
\RR_Q^\Psi (\xi) =\frac{Q^2}{\zeta(2)} G(\xi) +O_{\xi,\veps} ( Q^{23/12+\veps}) .
\end{equation}
Equality \eqref{7.28}  now follows from \eqref{7.30} and Corollary \ref{C8}.
\end{proof}

\section{Pair correlation of \texorpdfstring{$\{\theta_\gamma\}$}{theta[gamma] }
}\label{Section8}

\subsection{Proof of Theorem \ref{T1}}
In this section we pass to the pair correlation of the angles $\{\theta_\gamma\}$, estimating
\[
\RR^{\theta}_{Q}(\xi)  := \# \big\{ (\gamma,\gamma^\prime) \in \wMM_Q^2 :
0\leqslant Q^2 (\theta_{\gamma^\prime} -\theta_\gamma) \leqslant \xi\big\}.
\]
Define the pair correlation kernel $F(\xi,t)$ as follows
\begin{equation}\label{8.1}
F(\xi,t) = \sum_{M\in{\mathfrak S}} B_{M}(\xi,t) +
\sum\limits_{\substack{\ell \in [0,\xi) \\ K\in [1,\xi)}}
A_{K,\ell} (\xi,t).
\end{equation}
where  $B_{M}(\xi,t)$, $A_{K,\ell} (\xi,t)$ are the areas from from \eqref{7.20}, \eqref{7.27}, so that by
\eqref{7.30} we have
\[
\RR_Q^\Psi (\xi)=\frac{Q^2}{\zeta(2)}\int_0^{\pi /4} F(\xi,t) \frac{dt}{\cos^2 t} + O_{\xi,\veps} (
Q^{(11+c_0)/6+\veps}).
\]

\begin{proposition}\label{P15}
$\quad \displaystyle
\RR^\theta_Q (\xi)  =\frac{Q^2}{\zeta(2)} \int_0^{\pi/4} F\bigg(
\frac{\xi}{2\cos^2 t},t\bigg) \frac{dt}{\cos^2 t}
+O_{\xi,\veps} ( Q^{47/24+\veps}) .$
\end{proposition}
Before giving the proof, note that Theorem \ref{T1} follows from the proposition as $Q\rightarrow
\infty$, taking into account the different normalization in the definition of $\RR^\theta_Q (\xi)$,
$R_Q^{\mathfrak A} (\xi)$, and defining, in view of Proposition \ref{P15} and \eqref{8.1}:
\begin{equation*}
 B_M(\xi):= \int_0^{\pi/4} B_M\bigg(
\frac{\xi}{2\cos^2 t},t\bigg) \frac{dt}{\cos^2 t}, \qquad   A_{K,\ell}(\xi):=
\int_0^{\pi/4} A_{K,\ell}\bigg(
\frac{\xi}{2\cos^2 t},t\bigg) \frac{dt}{\cos^2 t}.
\end{equation*}
From the definitions of  $B_M(\xi,t)$, $A_{K,\ell}(\xi, t)$ in the equations following
\eqref{7.20}, \eqref{7.27}, it is clear that
$B_{M}(\frac{\xi}{2\cos^2 t},t )= B_M ( \frac{\xi}{2},0) \cos^2
t, A_{K,0}(\frac{\xi}{2\cos^2 t},t )=A_{K,0} ( \frac{\xi}{2},0)\cos^2 t,$ hence one has
\begin{equation}\label{8.2}
B_M(\xi)=\frac{\pi}{4}B_M \bigg( \frac{\xi}{2},0\bigg),\quad
A_{K,0}(\xi)=\frac{\pi}{4}A_{K,0} \bigg( \frac{\xi}{2},0\bigg),
\end{equation}
which together with \eqref{7.22} yields the formula for $B_M(\xi)$ given in Theorem \ref{T1}.
Note that the range of summation in Theorem \ref{T1} restricts to $K<\frac{\xi}{2}$, $
\ell< \frac{\xi}{2}$, compared with the range in \eqref{8.1}. Indeed, from the description of $A_{K,\ell}(
\frac{\xi}{2\cos^2 t},t)$ following \eqref{7.27} we see that
$\ell<\Upsilon_{\ell,K}\leqslant \frac{\xi}{2}$, while for $K$ we have
$
K<\frac{1}{L_{\ell-1}L_{\ell}}+\frac{K
L_\ell-K_{\ell-1}}{L_{\ell}}<\Upsilon_{\ell,K}\leqslant \frac{\xi}{2},
$ and similarly for $\ell=0$.
\begin{proof}
Consider $I=[\alpha,\beta)$ with $N=[Q^d]$,
$\vert I\vert=N^{-1} \sim Q^{-d}$,
$I^+=[\alpha-Q^{-d^\prime},\beta +Q^{-d^\prime}]$,
$I^-=[\alpha+Q^{-d^\prime},\beta -Q^{-d^\prime}]$ where $0< d=\frac{1}{24} < d^\prime=\frac{1}{12} < 1$.
Partition the interval $[0,1)$ into the union of $N$ intervals $I_j =[\alpha_j,\alpha_{j+1})$ with
$\vert I_j\vert =N^{-1}$ as above. Associate the intervals $I_j^\pm$ to $I_j$ as described above.
Denote
\begin{equation*}
\begin{split}
{\mathfrak R}^\sharp_Q  := & \{ (\gamma,\gamma^\prime) \in \wMM_Q^2 : \gamma \neq \gamma^\prime \} ,  \\
\RR^{\theta}_{I,Q}(\xi) := & \# \big\{ (\gamma,\gamma^\prime) \in {\mathfrak R}^\sharp_Q :
0\leqslant Q^2(\theta_{\gamma^\prime} -\theta_\gamma) \leqslant \xi,\  \Psi (\gamma),\Psi (\gamma^\prime)\in
I\big\} \\ &
\leqslant \RR_{I,Q}^{\theta,\natural}(\xi) := \# \big\{ (\gamma,\gamma^\prime) \in {\mathfrak R}^\sharp_Q :
0\leqslant Q^2( \theta_{\gamma^\prime} -\theta_\gamma) \leqslant \xi,\  \Psi (\gamma)\in I\big\} ,\\
\RR^\Psi_{I,Q} (\xi)  : = & \# \big\{ (\gamma,\gamma^\prime) \in {\mathfrak R}^\sharp_Q :
0\leqslant Q^2(\Psi (\gamma^\prime) -\Psi (\gamma)) \leqslant \xi,\
\Psi (\gamma),\Psi (\gamma^\prime)\in I \big\} ,\\
\RR^{\Psi,\flat}_{I,Q} (\xi) : =  & \# \big\{ (\gamma,\gamma^\prime) \in {\mathfrak R}^\sharp_Q :
0\leqslant Q^2( \Psi (\gamma^\prime) -\Psi (\gamma)) \leqslant \xi,\
\gamma_- ,\gamma_+ \in I \big\} ,  \\
\RR^{\Phi,\flat}_{I,Q} (\xi) : =  & \# \big\{ (\gamma,\gamma^\prime) \in {\mathfrak R}^\sharp_Q :
0\leqslant Q^2(\Phi (\gamma^\prime) -\Phi (\gamma)) \leqslant \xi,\
\gamma_- ,\gamma_+\in I \big\} .
\end{split}
\end{equation*}
Expressing $\theta_{\gamma^\prime}-\theta_\gamma$ and $\Psi (\gamma^\prime)-\Psi (\gamma)$ by the Mean Value
Theorem we find
\begin{equation}\label{8.3}
\RR_{I,Q}^\Psi \big( {\textstyle\frac{1}{2}} (1+\alpha^2)\xi \big) \leqslant \RR_{I,Q}^\theta (\xi) \leqslant
\RR_{I,Q}^\Psi \big( {\textstyle\frac{1}{2}} (1+\beta^2)\xi \big).
\end{equation}

\begin{lemma}\label{L16}
The following estimates hold:
\begin{enumerate}
\item[(i)]
$\displaystyle \ \sum_{j=1}^N \RR^\theta_{I_j,Q} (\xi) \leqslant \RR^\theta_Q (\xi) =\sum_{j=1}^N
\RR_{I_j,Q}^{\theta,\natural} (\xi) \leqslant
\sum_{j=1}^N \RR_{I_j^+,Q}^\theta (\xi) +O( Q^{15/8}\ln^2 Q).$
\item[(ii)] $\displaystyle \ \RR^\Psi_{I,Q}(\xi) =\RR_{I,Q}^{\Psi,\flat} (\xi) +O( Q^{1+d^\prime} \ln^2 Q)
.$
\end{enumerate}
\end{lemma}

\begin{proof}
The first inequality in (i) is trivial. For the second one note first that the total number of pairs
$(\gamma,\gamma^\prime)$ with
$0\leqslant \theta_{\gamma^\prime}-\theta_\gamma \leqslant \xi Q^{-2}$ and $qq^\prime \leqslant Q^{d^\prime}$, with
$\gamma_-=\frac{p}{q}$, $\gamma_+=\frac{p^\prime}{q^\prime}$ is $\ll_\xi Q^d (Q^{d^\prime}\ln Q) (Q\ln Q)$. For
$\gamma$ with $qq^\prime >Q^{-d^\prime}$
use $\Psi (\gamma^\prime)-\beta \leqslant \Psi (\gamma^\prime)-\Psi (\gamma) \leqslant
\frac{1}{qq^\prime} \leqslant Q^{-d^\prime}$, so $\Psi (\gamma^\prime) \in I_j^+$.
The proof of (ii) is analogous.
\end{proof}

Lemma \ref{L16} and \eqref{8.3} yield
\begin{equation*}
\sum_{j=1}^N \RR^\Psi_{I_j,Q} \big( {\textstyle\frac{1}{2}} (1+\alpha_j^2)\xi \big) \leqslant \RR^\theta_Q (\xi) \leqslant
\sum_{j=1}^N \RR^\Psi_{I_j^+,Q} \big( {\textstyle\frac{1}{2}} (1+\alpha_{j+1}^2)\xi \big) +O_\veps ( Q^{15/8+\veps}) .
\end{equation*}

To estimate $\RR^\Phi_{I,Q}(\xi)$ we repeat the previous arguments for a short interval $I$ as above.
Adding everywhere the condition $\gamma_-,\gamma_+\in I$ we modify $\RR^\Cap_Q$ by $\RR^\Cap_{I,Q}$
and $R^\Cap_Q$ by $R^\Cap_{I,Q}$ in Lemma \ref{L5},
$\RR^{\cap\,\cap}_Q$ by $\RR^{\cap\,\cap}_{I,Q}$ and
$R^{\cap\,\cap}_Q$ by $R^{\cap\,\cap}_{I,Q}$ in Lemma \ref{L6}.
The additional condition $\frac{p}{q},\frac{p^\prime}{q^\prime}\in I$ is inserted in \eqref{7.2}.
The condition $0\leqslant p^\prime \leqslant q^\prime$ is replaced by
$q^\prime \alpha \leqslant p^\prime < q^\prime \beta$ in \eqref{7.4}, and \eqref{7.24},
and $0\leqslant p\leqslant q$ is replaced by $q\alpha \leqslant p< q\beta$ in \eqref{7.4}.
The condition $v\in [0,q^\prime]$ is replaced by $v\in [q^\prime \alpha,q^\prime \beta)$ in the definition
of $\Omega_{M,q^\prime,\xi}$,  and
$\Omega^{\cap\,\cap}_{q^\prime,\ell,K,\xi}$.
The bodies $S_{M,\xi}$ and $T_{K,\ell,\xi}$ are substituted respectively by
$S_{I,M,\xi}$ and $T_{I,K,\ell,\xi}$ after replacing the condition $z\in [0,1]$ in their definition
by $z\in [\alpha,\beta)$. The analogs of \eqref{7.20} and \eqref{7.27} hold:
\begin{equation}\label{8.4}
\operatorname{Vol}(S_{I,M,\xi}) =\int_I B_{M}(\xi,t)\frac{dt}{\cos^2 t}
,\quad
\operatorname{Vol}(T_{I,K,l,\xi}) =\int_I
A_{K,\ell}(\xi, t)\frac{dt}{\cos^2  t}.
\end{equation}

The approach from Section \ref{Section7} under the changes specified in the previous paragraph leads to
\begin{equation}\label{8.5}
R_{I,Q}^{\Phi,\flat} (\xi) =R^\Cap_{I,Q}(\xi) +R^{\cap\,\cap}_{I,Q}(\xi) =
\frac{Q^2}{\zeta(2)} \int_{\tan^{-1} I} F(\xi,t) \frac{dt}{\cos^2 t}+O_{\xi,\veps} ( Q^{23/12+\veps}),
\end{equation}
with the pair correlation kernel $F(\xi,t)$ defined by \eqref{8.1}. We also have
\begin{equation}\label{8.6}
R^{\Phi,\flat}_{I^+,Q}(\xi) =R^{\Phi,\flat}_{I,Q} (\xi) +O_{\xi,\veps}(Q^{23/12+\veps}+Q^{2-d^\prime}) .
\end{equation}
The analogs of Lemmas \ref{L5}, \ref{L6} yield upon \eqref{8.5} and \eqref{8.6}
\begin{equation}\label{8.7}
\RR^{\Phi,\flat}_{I,Q}(\xi) =\frac{Q^2}{\zeta(2)} \int_{\tan^{-1} I} F\big( \xi+O(Q^{-1/3}),t\big) \frac{dt}{\cos^2
t}+O_{\xi,\veps} ( Q^{23/12+\veps}) =\RR^{\Phi,\flat}_{I^+,Q}(\xi) .
\end{equation}
The analog of Corollary \ref{C4} and \eqref{8.7} yield
\begin{equation}\label{8.8}
\begin{split}
\RR_{I,Q}^{\Psi,\flat} & (\xi) = \RR^{\Phi,\flat}_{I,Q} \big( \xi+O(Q^{-1/4})\big) +\RR^{\Phi,\flat}_{I,Q} \big(
O(Q^{-1/4})\big) +O( Q^{7/4+\veps})
\\ & =\frac{Q^2}{\zeta(2)} \int_{\tan^{-1} I} \Big( F\big( \xi+O(Q^{-1/4}),t\big)+F\big( Q^{-1/4},t)\big) \Big)
\frac{dt}{\cos^2 t} +O_{\xi,\veps} ( Q^{23/12+\veps} )
\\ & =\RR_{I^+,Q}^{\Psi,\flat} (\xi) .
\end{split}
\end{equation}
As shown in Section \ref{Section7} the function $F$ is $C^1$ in $\xi$, thus \eqref{8.8} gives actually\footnote{The
argument from Section \ref{Section7}
applies before integrating with respect to $t$ on $[ 0,\frac{\pi}{4}]$, showing that $F$ is $C^1$.}
\begin{equation}\label{8.9}
\RR_{I,Q}^{\Psi,\flat} (\xi) =\frac{Q^2}{\zeta(2)} \int_{\tan^{-1} I} F(\xi,t)\frac{dt}{\cos^2 t} +
O_{\xi,\veps} ( Q^{23/12+\veps}) = \RR^{\Psi,\flat}_{I^+,Q}(\xi) .
\end{equation}
Lemma \ref{L16} (i), \eqref{8.9}, and $F\in C^1 [0,\infty)$ yield
\begin{equation}\label{8.10}
\RR_{I,Q}^{\Psi} (\xi) = \frac{Q^2}{\zeta(2)} \int_{\tan^{-1} I} F( \xi ,t)
\frac{dt}{\cos^2 t} +O_{\xi,\veps} ( Q^{23/12+\veps} +Q^{2-d^\prime}) =\RR_{I^+,Q}^{\Psi} (\xi) .
\end{equation}
Let also $\omega_j=\tan^{-1} \alpha_j$. From \eqref{8.10} and \eqref{8.3} we further infer
\begin{equation*}
\begin{split}
\frac{Q^2}{\zeta(2)} \int_{\omega_j}^{\omega_{j+1}} & F\Big( {\textstyle\frac{1}{2}} (1+\alpha_j^2)\xi ,t\Big)
\frac{dt}{\cos^2 t} +
O_{\xi,\veps} ( Q^{23/12+\veps} +Q^{2-d^\prime}) \leqslant \RR^\theta_{I_j ,Q}(\xi) \leqslant
\RR^\theta_{I^+_j ,Q}(\xi) \\  & \leqslant
\frac{Q^2}{\zeta(2)} \int_{\omega_j}^{\omega_{j+1}} F\big( {\textstyle\frac{1}{2}} (1+\alpha_{j+1}^2)\xi ,t\big)
\frac{dt}{\cos^2 t} +
O_{\xi,\veps} ( Q^{23/12+\veps} +Q^{2-d^\prime}).
\end{split}
\end{equation*}
Employing also
\begin{equation*}
\int_{\omega_j}^{\omega_{j+1}} F\big( {\textstyle\frac{1}{2}} (1+\alpha_j^2)\xi,t\big) \frac{dt}{\cos^2 t} =
\int_{\omega_j}^{\omega_{j+1}} \Big( F\big( {\textstyle\frac{1}{2}} (1+\tan^2 t)\xi,t\big) +O(\omega_{j+1}-\omega_j)\Big)
\frac{dt}{\cos^2 t}
\end{equation*}
and $(\omega_{j+1}-\omega_j)^2 \leqslant Q^{-2d}$ we find
\begin{equation}\label{8.12}
\RR^\theta_{I_j,Q}(\xi) =\frac{Q^2}{\zeta(2)} \int_{\omega_j}^{\omega_{j+1}} F \big(
{\textstyle\frac{1}{2}} (1+\tan^2 t)\xi,t\big) \frac{dt}{\cos^2 t} +O_{\xi,\veps} ( Q^{23/12+\veps}) =
\RR^\theta_{I_j^+,Q}(\xi) .
\end{equation}
Finally Lemma \ref{L16} (i) and \eqref{8.12} yield the equality from Proposition \ref{P15}.
\end{proof}

\subsection{Explicit formula for \texorpdfstring{$g_2^{\mathfrak A}$}{g2} } \label{sec8.1}
Next we compute the derivatives $B_M'(\xi)$, thus proving Corollary \ref{C1}. We also obtain the
explicit formula \eqref{8.13} for $g_2^{\mathfrak A}$ on a larger range than in Corollary \ref{C1}, after
computing the derivative $A_{K,0}^\prime (\xi)$.

\begin{lemma} \label{L17}
For $M\in \mathfrak{S}$, let $T=T_M$, $Z=Z_M$ as in \eqref{3.1}. The derivative $B_M'(\xi)$ is
given by:
\[
B_M'(\xi)=\begin{cases} \vspace{0.1cm}
\displaystyle\frac{\pi}{4\xi^2}\ln\bigg(\frac{T+\sqrt{T^2-4}}{T+\sqrt{T^2-4-\xi^2}}\bigg)
& \text{\mbox if $\xi\leqslant 2Z$} \\  \vspace{0.1cm}
\displaystyle\frac{\pi}{8\xi^2}\ln\bigg(\frac{(T+\sqrt{T^2-4})^2
(T-\sqrt{T^2-4-\xi^2})}{(4+4Z^2)(T+\sqrt{T^2-4-\xi^2})} \bigg)
& \mbox{if $2Z \leqslant \xi\leqslant \sqrt{T^2-4}$} \\
\displaystyle\frac{\pi}{8\xi^2}\ln\bigg(\frac{(T+\sqrt{T^2-4})^2
}{4+4Z^2} \bigg)			 & \mbox{if  $\xi\geqslant \sqrt{T^2-4}.$}
          \end{cases}
\]
\end{lemma}
\begin{proof} Using \eqref{8.2}, we proceed as in the proof of Lemma \ref{L13}:
\[
B_M(\xi)=\frac{\pi}{4\xi}\int_{0}^{\pi/2}
\bigg( \frac{\xi}{\sqrt{T^2-4}}\cdot
\frac{1}{U_T+\cos(2\theta-\theta_M)}-\frac{|\sin (2\theta-\theta_M)|}{U_T+\cos
(2\theta-\theta_M)}\bigg)_+  d \theta,
\]
where $U_T=\frac{T}{\sqrt{T^2-4}}$ and $\theta_M\in (0,\frac{\pi}{2})$ has
$\sin\theta_M =\frac{2 Z}{\sqrt{T^2-4}}$.
Applying Lemma \ref{L11}, we obtain:
\[B_M'(\xi)=\frac{\pi}{4\xi^2}\int_I \frac{|\sin (2\theta-\theta_M)|}{U_T+\cos
(2\theta-\theta_M)}\, d\theta ,
\]
with $I=\{\theta\in (0,\frac{\pi}{2}):
|\sin(2\theta-\theta_M)|<\frac{\xi}{\sqrt{T^2-4}}\}$. Clearly
$I=( 0,\frac{\pi}{2})$ when $\xi>\sqrt{T^2-4}$, and if $\xi\leqslant
\sqrt{T^2-4}$, let $\alpha=\alpha(\xi)\in ( 0,\frac{\pi}{4})$ such that $\sin
2\alpha= \frac{\xi}{\sqrt{T^2-4}}$. Then
\[
\xi\leqslant 2Z \iff
\alpha\leqslant \theta_M/2 \iff I=[\theta_M/2-\alpha, \theta_M/2+\alpha],\]
\[
2Z\leqslant \xi \leqslant \sqrt{T^2-4}\iff \alpha\in[\theta_M/2, \pi/4] \iff I=[0,
\theta_M/2+\alpha]\cup [\pi/2+\theta_M/2-\alpha , \pi/2],
\]
and the integral is easy to compute. For $M= \left( \begin{smallmatrix} 1 & 1 \\ 0 & 1
\end{smallmatrix}\right)$ and $\xi=3$, the region with area  $B_M(\xi/2,0)$ is the one
hashed vertically in Figure \ref{Figure6}.
\end{proof}
\comment{
\begin{figure}[ht]
\begin{center}
\includegraphics*[scale=0.55, bb = 0 0 350 270]{sect1.eps}
\includegraphics*[scale=0.55, bb = 0 0 350 270]{sect2.eps}
\caption{The area $B_M(\xi/2,0)$ when $2Z\leqslant \xi \leqslant \sqrt{T^2-4}$.}
\label{Figure6}
\end{center}
\end{figure}
}
A similar computation using \eqref{8.2} gives the formula:
\[A_{K,0}'(\xi)=\frac{\pi}{4\xi^2} \cdot \begin{cases} 0 & \text{ if } \xi \leqslant 2K \\
\ln(1+K^2)+\ln\left(\displaystyle\frac{(1+x_1^2)(1+(x_2-K)^2)}{(1+x_2^2)(1+(x_1-K)^2)}
\right)& \text{ if } \xi\in [2K, K\sqrt{K^2+4}]\\
\ln(1+K^2) & \text{ if } \xi\geqslant K\sqrt{K^2+4} ,
\end{cases}
\]
where $x_2>x_1$ are the roots of $x^2(\xi+2K)-2xK(\xi+K)+\xi(K^2+1)-2K=0$. By the last
paragraph in Remark \ref{R14}, the body $T_{1,1,\xi}$ is empty, so $A_{1,1}(\xi)=0$,
and we have an explicit formula on a larger range than in the introduction:
\begin{equation}\label{8.13}
g_2^{\mathfrak A}\left( \frac{3}{4\pi}\xi
\right)=\frac{32\pi}{9\zeta(2)}\left(\sum_{M\in
\mathfrak{S}} B_M'(\xi)+ A_{1,0}'(\xi)\right), \quad
0<\xi\leqslant 4.
\end{equation}

We can now explain the presence of the spikes in the graph of
$g_2^{\mathfrak{A}}$ in Figure \ref{Figure1}. The function $B_M'(\xi)$ is not
differentiable
at $\xi=2F$ and $\sqrt{T^2-4}$, while the function $A_{K,0}'(\xi)$ is not
differentiable at $\xi=2K$ and $\sqrt{(K^2+2)^2-4}$. At the point $\xi=\sqrt{5}$, two
of the functions $B_M'(\xi)$, as well as $A_{1,0}'(\xi)$, have infinite slopes on the
left, which gives the spike on the graph of $g_2^{\mathfrak A}(x)$ at
$x=\frac{3}{4\pi}\sqrt{5}$.

\bigskip

\textbf{\small Acknowledgments}
We would like to thank the referees for their constructive comments, which led to
improvements and clarifications in the presentation. The first author was supported in
part by the CNCS - UEFISCDI grant PN-II-RU-TE-2011-3-0259. The second author was supported
in part by grant CNCSIS PD 243-171/2010. The third author was supported in part by
European Community's Marie Curie
grant PIRG05-GA-2009-248569 and by the CNCS - UEFISCDI grant PN-II-RU-TE-2011-3-0259. The
fourth author was supported in part by NSF grant DMS-0901621.

\end{document}